\documentclass[12pt, reqno]{amsart}

\usepackage[margin=1.2in]{geometry}
\usepackage{amssymb,amsmath,exscale,latexsym,amsthm,graphics,graphics,mathtools}
\usepackage{epsfig}    
\usepackage{epic}      
\usepackage{eepic}     
\usepackage[matrix,arrow,curve,cmtip]{xy} 
                 
\usepackage{amsfonts,amsbsy}
\usepackage{enumerate} 
\usepackage{paralist}
\usepackage{xcolor}
\usepackage{mathrsfs}
\usepackage{enumitem}
\usepackage{verbatim}  
\usepackage{tikz-cd}
\usepackage{MnSymbol} 
\usepackage{bbm}
\usepackage{hyperref}
\hypersetup{%
  colorlinks=false,
  urlbordercolor=cyan, linkbordercolor=cyan}

\theoremstyle{plain}
        \newtheorem{theorem}{Theorem}[section]
        
        \newtheorem{lemma}[theorem]{Lemma}
        
        \newtheorem{prop}[theorem]{Proposition}

\theoremstyle{definition}
        \newtheorem{definition}[theorem]{Definition}
        \newtheorem{remark}[theorem]{Remark}

        \newtheorem{example}[theorem]{Example}
     
\numberwithin{equation}{section}

\def\reminder #1 {{\sf #1}}
\def\hide #1 {}
\long\def\longhide #1 {}

\newcommand{\diam}  {\operatorname{diam}}

\newcommand{\R}{\mathbb{R}}     
\newcommand{\Hyp}{\mathbb{H}} 
\newcommand{\Haus}{\mathscr{H}} 
  
\newcommand{\Sph}{\mathbb{S}}      
\newcommand{\RP}{\mathbb{RP}}      
\newcommand{\C}{\mathbb{C}}      
\newcommand{\N}{\mathbb{N}}

\newcommand{\dm}{\mathrm{d}}   
\newcommand{\Diff}{\mathrm{d}}   
\newcommand{\norm}[1]{\left|#1\right|}

\newcommand{\vspan}{\operatorname{span}}

\newcommand{\lm}{\varnothing}

\newcommand{\ii}{\mathbbm{i}}

\newcommand{\Mob}{\operatorname{M\ddot{o}b}}
\newcommand{\id}{\operatorname{Id}}
\renewcommand{\Re}{\operatorname{Re}}
\renewcommand{\Im}{\operatorname{Im}}

\newcommand{\twovector}[2]{
\left[\begin{array}{c}
#1\\#2	
\end{array}\right]}

\begin{document}

\title{Projection theorems for \\[3pt] linear-fractional families of projections}

\author{Annina Iseli and Anton Lukyanenko}

\address {Department of Mathematics, University of Fribourg, 1700 Fribourg,  Switzerland.}
\address {Department of Mathematics, George Mason University,   Fairfax, VA 22030}

\email{annina.iseli@unifr.ch}
\email{alukyane@gmu.edu}

\maketitle

\begin{abstract}
Marstrand's theorem states that applying a generic rotation to a planar set $A$ before projecting it orthogonally to the $x$-axis almost surely gives an image with the maximal possible dimension $\min(1, \dim A)$. We first prove, using the transversality theory of Peres-Schlag locally, that the same result holds when applying a generic complex linear-fractional transformation in $PSL(2,\C)$ or a generic real linear-fractional transformation in $PGL(3,\R)$. We next show that, under some necessary technical assumptions, transversality locally holds for restricted families of projections corresponding to one-dimensional subgroups of $PSL(2,\C)$ or $PGL(3,\R)$. Third, we demonstrate, in any dimension, local transversality and resulting projection statements for the families of closest-point projections to totally-geodesic subspaces of hyperbolic and spherical geometries.
\end{abstract}

\vspace{6pt}

\section{Introduction}
\label{sec:introduction}

\setcounter{equation}{0}

Research on projection theorems in various spaces has a long-standing tradition in geometric measure theory. Perhaps the earliest work in the field is due to Besicovich \cite{Besic1939} and Federer \cite{Federer1947}, who characterized rectifiable sets in $\R^n$ in terms of the Hausdorff measure of their image under orthogonal linear projections.
Inspired by their work, Marstrand \cite{Marstrand1954} initiated a more extensive analysis of the effect of orthogonal linear projections on the Hausdorff measure and dimension of Borel sets, showing that the image of a planar set $A$ under almost every orthogonal projection has the maximal possible Hausdorff dimension. More precisely: let $A$ be any Borel set, and consider, for each $\theta\in [0,2\pi)$ the orthogonal projection $P_\theta(A):\R^2\rightarrow L_\theta$, where $L_\theta$ is a line at angle $\theta$ to the $x$-axis. Since each $P_\theta$ is Lipschitz into a $1$-dimensional target space, $\dim(P_\theta A) \leq \min\{1, \dim A\}$. Marstrand's Theorem states that the upper bound is, in fact, attained for $\Haus^1$-almost every angle $\theta$, where $\Haus^s$ denotes the Hausdorff $s$-measure.

Marstrand's planar result was subsequently generalized, in various respects, to orthogonal projections of $\R^n$ onto $m$-planes by Kaufman \cite{Kaufman1968}, Mattila \cite{Mattila1975}, and Falconer \cite{Falconer1982}. More recently, these results have been generalized and extended to further spaces with natural projection families, including the families of horizontal and vertical projections in the Heisenberg groups \cite{BFMT2012,BDCFMT2013,Hovila2014,Harris2020}, and certain families of closest-point projections in hyperbolic $n$-space $\Hyp^n$ and the $2$-sphere $\Sph^2$ \cite{BaloghIseli2016,BaloghIseli2018}.

It is often straightforward to extend projection theorems from a small projection family to a larger one that contains it, see e.g.~Proposition \ref{prop:enlargedFamily}. Conversely, restricting projection theorems to a measure-zero subfamily is not always possible (e.g.~ the family of orthogonal projections from $\R^3$ to lines \emph{in the XY plane} always maps the Z axis to a single point). Järvenpää et al.~\cite{JJLL2008} introduced the notion of \emph{restricted families of projections} and provided conditions  under which projection theorems hold for a one-dimensional family of projections that is induced by a curve in the Grassmanian $G(n,m)$ of $m$-planes in $\R^n$. Identifying more general conditions under which a restricted family retains projection theorems remains an actively-studied task \cite{OrpVen, FassOrp2014,Chen2018,Harris_Arx2021}. 

Marstrand's original proof of his theorem used mainly methods from planar geometry. The generalizations by Kaufman and Mattila employed potential theory. Some decades later, a further developed version of such potential theoretic methods allowed Peres-Schlag \cite{PS2000} to establish a powerful result that links dimension preservation for families of mappings with a more-easily verified \emph{(differentiable) transversality condition} (see \S \ref{subsec:PrelimTransversality}). Establishing (local) transversality has become a standard method for proving projection theorems \cite{BaloghIseli2016,Hovila2014,Mattila2019}.  Theorem \ref{thm:PS_n} below summarizes the consequences of local transversality in our setting. Furthermore, its control of the distortion of Sobolev dimension of measures under the given family has been applied e.g.\ in the study of  distance set problems \cite[\S 5]{Mattila2004} and \cite[\S 8]{PS2000}, intersections of Cantor sets \cite{PeresSimonSolomyak2003}, and Bernoulli convolutions \cite[\S 5]{PS2000}. 

Combining the above themes, we will be interested in using transversality to prove projection theorems for certain one-dimensional families of projections in a non-Euclidean setting. Namely, we will replace the rotationally-symmetric projections in Marstrand's Theorem with projection families which arise from linear-fractional symmetries of the Riemann sphere $\hat \C$ and the real projective plane $\RP^2$, and characterize the regions on which transversality and projection theorems hold. Additionally, our use of projective geometry will allow us to quickly obtain transversality and projection results in spherical and hyperbolic spaces, extending the results of \cite{BaloghIseli2016,BaloghIseli2018}, see also \cite{dufloux2018linear}.

\subsection{Projection families induced by group actions}

Marstrand's theorem fixes a set $A\subset \R^2$ and varies the mapping $P_\theta$. Equivalently, one can first rotate $A$ by a rotation $R_\theta\in O(2)$ and then project it to $\R$ by the fixed mapping $\pi(x,y)=x$. This phrasing emphasizes the role of the rotation family $O(2)$. We ask whether projection theory extends to cases where $O(2)$ is replaced with another group action that arises naturally in a geometric setting.

We consider the following general framework:

\begin{definition}
Let $N,M$ be smooth manifolds, $S_0\subset N$ a closed subset, $\pi: N \rightarrow M$ a mapping (called a \emph{projection}) with domain $N\setminus S_0$, and $G$ a group acting on $N$ by $(g,p)\mapsto g(p)$. Then the \emph{projection family $\Pi$ induced by $\pi$ and the action of $G$ on $N$} is given by
$$\Pi: G\times N \rightarrow M, \ \  (g,p)\mapsto \pi(g(p))$$
on its domain $(G\times N)\setminus S$, where $S=\{(g, p): g(p)\in S_0\}$.
We refer to a family $\Pi$ arising in this manner as \emph{induced by a group action}. Such families are characterized by the condition $\Pi(g, p)=\Pi(\id, g(p))$.
\end{definition}

We will be predominantly interested in the case where $N$ is $\R^2$ or its more symmetric compactifications $\hat \C$ and $\RP^2$. We will analyze several classes of projection families induced by group actions by establishing local transversality and, where it fails, using more direct arguments that allow us to draw the same conclusions concerning dimension preservation. These are stated in Theorem \ref{thm:PS_n} and include the analogs of all classical results about orthogonal projections in $\R^n$ such as the Marstrand and Besicovich-Federer projection theorems.

\begin{definition}
\label{defi:satisfiesProjectionTheorems}
We say that a projection family $\Pi$ \emph{satisfies projection theorems} if the conclusions of Theorem \ref{thm:PS_n} hold on the domain of $\Pi$.
\end{definition}

More generally, it is natural to ask under what conditions a family $\Pi$ induced by the action of $G$ on $M$ and  a mapping $\pi: N\rightarrow M$ is (locally) transversal or satisfies projection theorems, when one equips $N, M$ with Riemannian metrics and volume measures, and $G$ with its Haar measure.
Clearly, if a fiber is $G$-invariant then projection theorems must fail along this fiber; and it seems intuitively plausible that projection theorems should hold elsewhere. However, as we show below, proving this intuition is not always straightforward even in specific examples, and in fact, the intuition fails when applied to transversality.

\subsection{Results}
In this paper, we focus on planar projection theory, extended to two separate spaces: $\hat \C=\C\cup\{\infty\}$ with the group of complex linear-fractional mappings (M\"obius transformations), or the projective plane $\RP^2$ with the group of real linear-fractional mappings (projective transformations).  Additionally, the projective geometry perspective allows us to quickly analyze closest-point projections in hyperbolic and spherical geometries.

We first study complex linear fractional transformations, denoted by $\Mob$, acting on the Riemann sphere $\hat \C=\C\cup\{\infty\}$, which is a natural family of motions to consider from the viewpoint conformal geometry and complex analysis. M\"obius transformations have the form
\[z \mapsto \frac{a z+b}{cz+d}, \text{ where $a,b,c,d\in \C$ with }ad-bc=1,\]
and can be identified with the group of complex determinant-one matrices $SL(2,\C)$. Using a decomposition of $\Mob$ into $O(2)$ and a complementary manifold\footnote{As a manifold, $\Mob$ globally decomposes as  $\Mob=O(3)\times \R\times\R^2$ due to the Iwasawa decomposition, which is then diffeomorphic to $\Mob=O(2)\times S^2\times \R^3$.}, we establish:
\begin{theorem}
\label{thm:MobAllProjections}
The family $\Pi: \Mob\times \hat \C\rightarrow \R$ given by $\Pi(g, z)=\Re(g(z))$ is locally transversal and therefore satisfies projection theorems on its domain $\Mob\times \hat \C\setminus \{(g, g^{-1}(\infty)): g\in \Mob\}$.
\end{theorem}
We ask what restricted families in $\Pi: \Mob\times \hat \C\rightarrow \R$ satisfy projection theorems. For some symmetric families the answer is well-known: the family corresponding to the rotations $O(2)$ satisfies projection theorems, while families corresponding to dilations or translations are non-transversal since they simply rearrange the fibres of the projection. For other families, such as $(\theta, z) \mapsto \Re\left( \frac{\cos(\theta) z - \sin (\theta)}{\sin(\theta)z + \cos(\theta)}\right)$, the answer is not 
obvious.  We prove:

\begin{theorem}
\label{thm:MobSpecialProjections}
Let $\Gamma\subset \Mob$ be a one-dimensional Lie subgroup, and $\Pi: \Gamma\times \hat \C\rightarrow \R$ the family given by $\Pi(\gamma,z)=\Re(\gamma(z))$. Then $\Pi$ satisfies projection theorems on its domain, with the following natural exceptions:
\begin{enumerate}[label={(\arabic*)\ }, itemsep=2pt,topsep=1pt]
    \item If $\Gamma$ consists of Euclidean dilations and translations, then projection theorems fail globally.
    \item If the orbit $\Gamma(\infty)$ is a vertical line, then projection theorems fail along this line.
\end{enumerate}
\end{theorem} 
While the projection theorems confirm the expected behavior, the underlying transversality result has an artefact that we work around: transversality fails along a set corresponding to the linearization of $\Gamma(\infty)$, see Theorem \ref{thm:MobTransversality}.

\begin{figure}
    \centering
    \hfill{}
    \includegraphics[width=.35\textwidth]{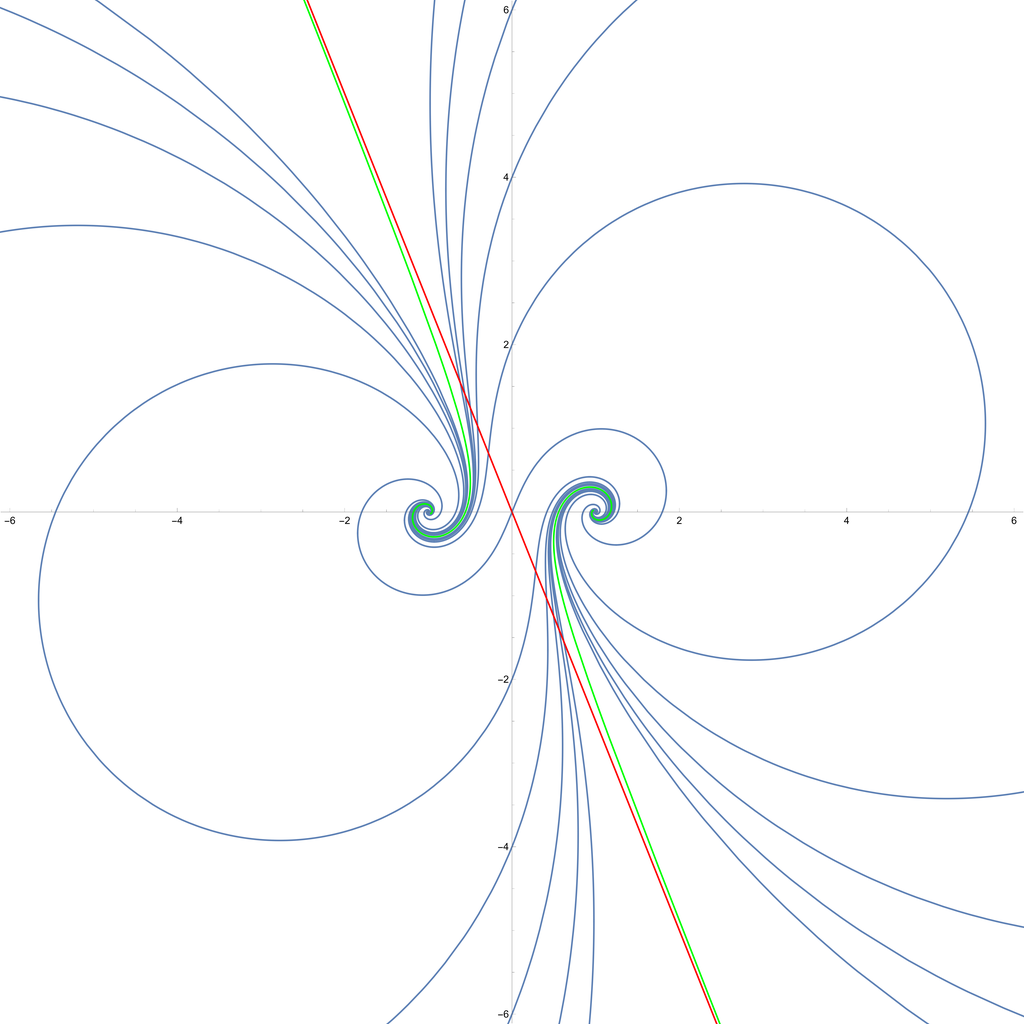}
    \hfill{}
    \includegraphics[width=.35\textwidth]{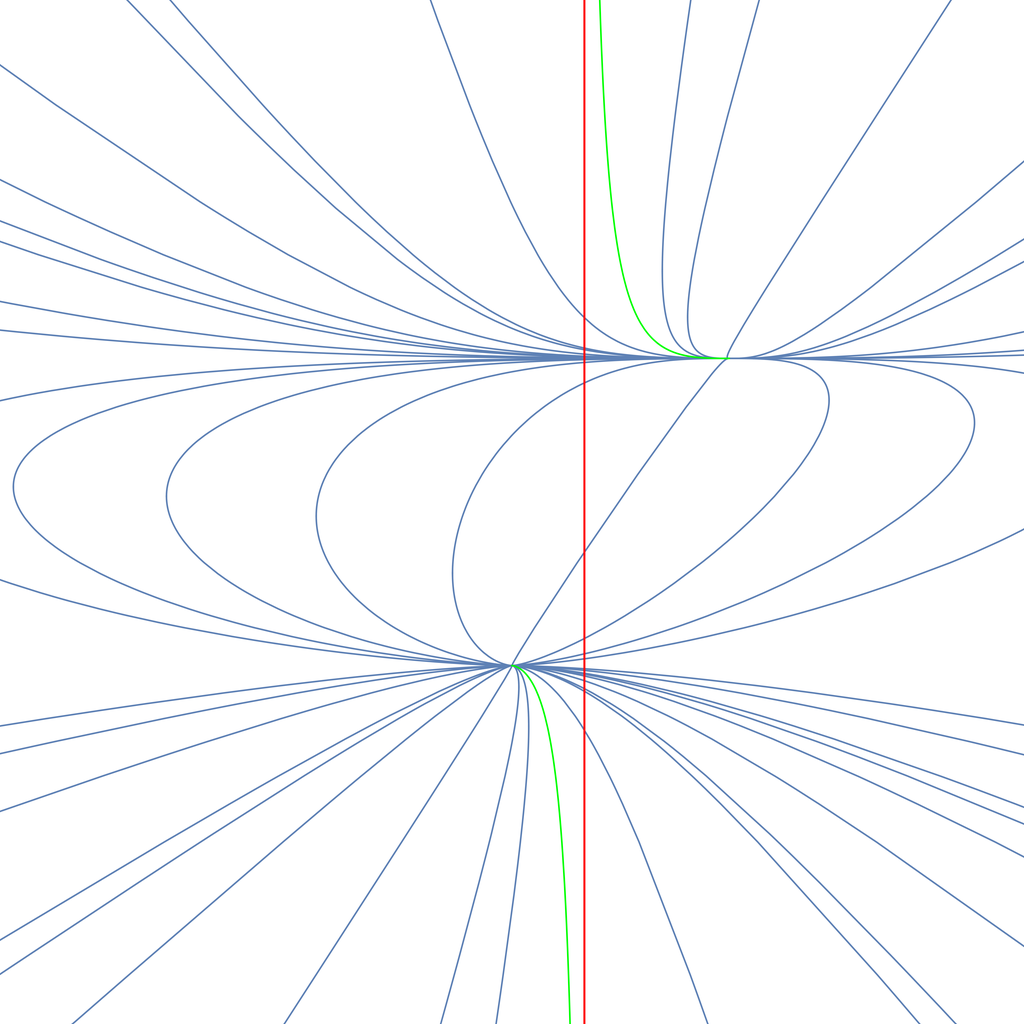}
    \hfill{}
    \caption{Groups $\Gamma$ in $\Mob$, left picture, and in $GL(3,\R)$, right picture, illustrated via blue orbits. Projection theorems hold on the full domain of the associated family $\Pi$, but transversality fails at $\gamma=\id$ along the red line $L$ tangent to the green orbit $\Gamma(\infty)$ (resp., $\Gamma(\infty_Y)$), and more generally on the set $\{(\gamma, p):\gamma^{-1} p\in L\}\subset \Gamma\times \hat C$ (resp., $\Gamma \times \RP^2$).
    }
    \label{fig:infiniteTangent}
\end{figure}

We next consider real linear fractional (projective) transformations $PGL(3,\R)$ acting on the real projective plane $\RP^2$, which is the natural family of motions to consider from the viewpoint of projective geometry. The projective plane $\RP^2$ is a compactification of $\R^2$ that distinguishes points at infinity corresponding to linear directions. In $\RP^2$, the projection $\pi:(x,y)\mapsto (x,0)$ can be interpreted as a linear point-source projection from the infinite point $\infty_Y$ in the $Y$ direction, and naturally extends to a mapping $\pi: \RP^2 \rightarrow \RP^1 = \R \cup\{\infty_X\}$. Projective transformations have the form
\[(x,y)\mapsto \left(\frac{a_{11}x+a_{12}y+a_{13}}{a_{31}x+a_{32}y+a_{33}},\frac{a_{21}x+a_{22}y+a_{23}}{a_{31}x+a_{32}y+a_{33}}\right),\text{ where }(a_{i,j})\in GL(3,\R).
\]

We prove analogs of Theorems \ref{thm:MobAllProjections} and \ref{thm:MobSpecialProjections}:

\begin{theorem}
\label{thm:RPAllProjections}
The family $\Pi: GL(3,\R) \times \RP^2 \rightarrow \RP^1$ given by $\Pi(g, z)=\pi(g(z))$  on its domain  $ (GL(3,\R) \times \RP^2 )\setminus \{(g, g^{-1}(\infty_Y)): g\in GL(3,\R)\}$ is locally transversal and therefore satisfies projection theorems.

\end{theorem}

\begin{theorem}
\label{thm:RPSpecialProjections}
Let $\Gamma\subset GL(3,\R)$ be a one-dimensional Lie subgroup, and let $\Pi: \Gamma\times  \RP^2 \rightarrow \RP^1$ the family given by $\Pi(\gamma,p)=\pi(\gamma(p))$. Then $\Pi$ satisfies projection theorems on its domain  $ (GL(3,\R) \times \RP^2 )\setminus \{(g, g^{-1}(\infty_Y)): g\in GL(3,\R)\}$, with the following natural exceptions:
\begin{enumerate}[label={(\arabic*)\ }, itemsep=2pt,topsep=1pt]
    \item If $\Gamma$ consists of mappings of the form $(x,y)\mapsto \left(\frac{a_{11}x+a_{13}}{a_{31}x+a_{33}},\frac{a_{21}x+a_{22}y+a_{23}}{a_{31}x+a_{32}y+a_{33}}\right)$, or, equivalently, preserves the source $\infty_Y$ of the projection $\pi$, then projection theorems fail globally,
    \item If the orbit $\Gamma(\infty_Y)$ is a vertical line or the line at infinity, then projection theorems fail along this line.
\end{enumerate}
\end{theorem}
As for Theorem \ref{thm:MobSpecialProjections}, the behavior in Theorem \ref{thm:RPSpecialProjections} is expected, but transversality fails along the linearization of the orbit $\Gamma(\infty_Y)$, see Theorem \ref{thm:RPTransversality}. Importantly, this non-transversality locus may be the line at infinity $\RP^2\setminus \R^2$. In Example \ref{ex:point-source}, we consider a group $\Gamma\subset GL(3,\R)$ that is conjugate to $O(2)$ by a mapping that sends $\pi$ to the point-light-source projection mapping, and the non-transversality at infinity for the family associated to $\Gamma$ turns into non-transversality along a line in $\R^2$ that is tangent to the orbit of the light source under $O(2)$.

We finish by applying projective geometry to closest-point projections in hyperbolic space $\Hyp^n$ and spherical space $\Sph^n$. In each of these spaces, we consider the families of $m$-dimensional totally-geodesic subspaces through a fixed point, and the corresponding families of closest-point projections. (Note that each of these families can be identified with the Grassmannian $G(n,m)$ via the exponential mapping.) Finding an appropriate change of coordinates based on projective geometry, we show that each of these families is, in fact, equivalent to the Euclidean orthogonal projection families, and is therefore transversal: 

\begin{theorem}
\label{thm:hypsphintro}
Let $X$ be hyperbolic space $\Hyp^n$ or the sphere $\Sph^n$, and $p$ a point in $X$. In the spherical case, let $S^{n-1}\subset \Sph^n$ denote the great sphere perpendicular to $p$. Let $\Pi: G(n,m)\times X\rightarrow \R^m$ be the family of nearest-point projections onto $m$-dimensional totally geodesic subspaces of $X$ through $p$.  Then $\Pi$ is locally transversal and satisfies projection theorems on all of $\Hyp^n$, respectively on $\Sph^n\setminus S^{n-1}$.
\end{theorem}
See Theorem \ref{thm:curvedTransversality} for a more precise phrasing of this result.

Previously to this work, transversality was known for $\Hyp^2$ and $\Sph^2$ \cite{BaloghIseli2016}, and projection theorems \cite{BaloghIseli2018} and a weaker version of transversality \cite{AnninaPhD} were known in $\Hyp^n$ for $n\geq 3$ but not for $\Sph^n$.

\subsection{Methods}
The proofs of the above results are based on direct calculations establishing transversality, combined with appropriate geometric and Lie-theoretic considerations for $\Mob$ and $GL(3,\R)$. We establish the following auxiliary results to aid with the calculations, which are of independent interest and rely on our strong regularity and symmetry assumptions: 
\begin{itemize}[itemsep=2pt,topsep=1pt]
    \item Lemma \ref{lemma:transversalitycomposition} allows us to change coordinates and reparametrize $C^2$ projection families, in particular allowing us to speak of local transversality in coordinate charts of manifolds, as in the definition of projection families induced by group actions;
    \item Lemma \ref{lemma:thetaZero} allows us to check transversality of symmetric $C^2$ projection families $\Pi: G\times N\rightarrow \R$ only along the identity of $G$ rather than in specified neighborhoods;
    \item Proposition \ref{prop:enlargedFamily} shows that transversality persists when one enlarges the parameter space of a symmetric $C^2$ projection family, putting a special focus on identifying minimal transversal families and leads into our study of symmetric projection families with one-parameter symmetry groups.
\end{itemize}

\subsection{Future Directions}
Our results can be seen as a first step in the study of projection theory for projection families induced by group actions. There are several immediate generalizations one can ask about, by working with new projections (e.g.~point-source or ultraparallel), new groups (e.g.~quasi-conformal), and new spaces (e.g.~$\R^n$, complex hyperbolic space, or Heisenberg groups). In higher-dimensional spaces, this would require, in view of Proposition \ref{prop:enlargedFamily}, understanding \emph{minimal} transversal families. In Heisenberg groups, projection theorems generally fail for the projection family generated by isometries and either vertical or horizontal projections, but can be expected to hold for the group of M\"obius transformations.

\subsection{Structure of the paper}
The paper is structured as follows:
Section \ref{sec:prelims} is for preliminaries. In \ref{subsec:LieGrass}, we recall some basics on Lie groups and Grassmannians. This subsection can safely be skipped by readers who are familiar with the terminology of Lie groups.  
In subsections \ref{subsec:PrelimTransversality} - \ref{sec:trans_for_groups}, we recall the formal definition of transversality and its implications. Moreover, we prove preliminary results for families of of projections induced by group actions. In particular, we prove that transversality can be lifted to families of projections induced by larger Lie groups.
In Section \ref{sec:season1}, we study projection theory in the plane for projections induced by M\"obius transformations. In particular, we prove Theorems \ref{thm:MobAllProjections} and \ref{thm:MobSpecialProjections}. In Section \ref{sec:season2}, we study projection theory in the plane for projections induced by real projective transformations. In particular, we prove Theorems \ref{thm:RPAllProjections} and \ref{thm:RPSpecialProjections}. In Section \ref{sec:hypsph}, we consider closest-point projections in spherical and hyperbolic geometry, proving Theorem \ref{thm:hypsphintro}.

\section{Preliminaries}
\label{sec:prelims}

\subsection{Lie groups and Grassmannians}\label{subsec:LieGrass}
We now briefly recall the basic structure and introduce the notation for Lie groups, their quotients, and the particular case of the orthogonal group and Grassmannians. For a detailed account, see the textbooks \cite{Lee2013,Knapp2002} and \cite[Ch.~3]{Mattila1995}.

A Lie group $G$ is a smooth manifold with a group structure such that the product mapping $G\times G\rightarrow G$ and the inversion mapping $G\rightarrow G$ are $C^\infty$-mappings. As a Lie group, $G$ possesses a unique (up to rescaling) measure that is invariant under left multiplication, called the (left) Haar measure. Likewise, we have an (effectively) unique Riemannian metric on~$G$, given by choosing an inner product on the tangent space of the identity and translating it to the full tangent bundle by left translations; different choices of inner product yield bi-Lipschitz metrics. Since the Haar measure is smooth (i.e.~given in charts by integrating a smooth function), it agrees with the Hausdorff measure of the appropriate dimension, and the notion of zero-measure is compatible with Lebesgue measure when $G$ is viewed in charts.

The exponential map 
$\exp : T_{\id} G\rightarrow G$ is a smooth mapping that relates the Lie algebra structure of the tangent space at the identity with the Lie group structure of $G$. For matrix groups, the exponential map can be written explicitly using the Taylor expansion of $e^x$, interpreted appropriately for matrices. On a small neighborhood of $0$, $\exp$ is a diffeomorphism onto an open neighborhood of ${\id}\in G$. Via an identification of $T_{\id}G$ with $\R^n$, this provides exponential coordinates on $G$ near $\id$. Further composing with any element $g\in G$ provides exponential coordinates near $g$.

A Lie subgroup $\Gamma \subset G$ is a subset that is both an immersed submanifold and a subgroup of $G$. Connected one-dimensional subgroups are images of lines through the origin in $T_{\id} G$ under the exponential map.
In studying quotients of $G$ by $\Gamma$, one often imposes the condition that $\Gamma$ is closed as a subset of $G$ (non-closed examples such as the irrational line on the torus do not behave well in charts and do not provide good quotient spaces). In exponential coordinates, one locally sees $\Gamma$ as a linear subspace of $G$, concluding that the inclusion $\iota: \Gamma\hookrightarrow G$ is a smooth embedding. Translates of this subspace under the action of $G$ provide a smooth foliation, and a smooth section of this foliation provides a coordinate chart for the quotient (coset) space $G/\Gamma$. One concludes that $G/\Gamma$ is a manifold, and obtains in particular a smooth family of mappings $g\in G$ that relate each point of $G/\Gamma$ to the basepoint $[\id]\in G/\Gamma$. 

The orthogonal group $O(n)$ consists of matrices $M$ satisfying $M^\mathsf{T}M=\id$, with Lie algebra $\mathfrak{o}(n)$ consisting of skew-symmetric matrices. The standard basis elements of $\mathfrak o(n)$ are the matrices $A_{i,j}$ for $i<j$ with a 1 in the $(i,j)$ entry, a $(-1)$ in the $(j,i)$ entry, and $0$s elsewhere. Under the exponential mapping, the mapping $R_{\theta_{i,j}}^{i,j}=\exp(\theta_{i,j}A_{i,j})$ acts by rotation on $\operatorname{span}(e_i,e_j)\subset\R^n$, and by identity on the remaining basis vectors of $\R^n$.

The Grassmannian $G(n,m)$ consists of all $m$-dimensional subspaces ($m$-planes) through the origin in $\R^n$. Fixing a specific $m$-plane $V_0$ (say, spanned by the first $m$ basis vectors), one observes that mappings in $O(n)$ send $V_0$ to other elements of $G(n,m)$ and that the orbit of $V_0$ under $O(n)$ is all of $G(n,m)$. One can then identify $G(n,m)$ with the quotient space $O(n)/(O(m)\times O(n-m))$ by observing that $O(m)\times O(n-m)$ is the maximal subgroup that leaves $V_0$ in place. One can give $G(n,m)$ a manifold structure explicitly by considering the possible basic rotations near the identity. Note first that the rotations $R^{i,j}_{\theta_{i,j}}$ for $i,j\leq m$ rotate $V_0$ within itself, the rotations $R^{i,j}_{\theta_{i,j}}$ for $i,j> m$ leave $V_0$ pointwise-invariant, and each of the remaining basic rotations $R^{i,j}_{\theta_{i,j}}$ moves $V_0$ into a new direction. 

Let $W=\{(\theta_{i,j})_{i,j}\in \mathfrak{o}(n): \theta_{i,j}=0 \text{ if } i,j\leq m \text{ or } i,j>m\}\subset \mathfrak o(n)$.
The exponential map identifies a neighborhood of $W$ near the identity with a neighborhood of $V_0\in G(n,m)$, and the mapping $\exp((\theta_{i,j})_{i,j})\in O(n)$ provides a smoothly-varying identification of $\exp((\theta_{i,j})_{i,j})V_0$ with $V_0$. One may further identify $V_0$ with $\R^m$ to obtain a local smoothly-varying identification of elements of $G(n,m)$ with $\R^m$. 

We comment further that the exponential mapping on a Riemannian manifold $M$ based at a point $p\in M$ identifies the Grassmannian $G(n,m)$ with a family of smooth submanifolds of $M$ meeting at $p$, by a $C^\infty$ mapping that is a priori only locally defined. In $\mathrm{CAT}(0)$ spaces such as real hyperbolic space, the exponential mapping is in fact a diffeomorphism between the tangent space and the manifold.

\subsection{Transversality and projection theorems}\label{subsec:PrelimTransversality}

We now define transversality and state the primary consequences of transversality for projection theory, phrased in a language suitable for geometric applications. For details, see \cite[Theorem 4.9]{PS2000}, \cite[Ch.\,5]{Mattila2004} and \cite[Ch.\,5,18]{Mattila2015}.

Let $n,m,k\in \N$ with $k\geq m$. 
A ($k$-parameter) \emph{family of mappings}, is a continuous mapping $\Pi:\Lambda\times \Omega\to \R^m$ where  $\Omega\subset \R^n$ is open and $\Lambda\subset \R^k$ is open.We assume, furthermore, that $\Pi$ is $C^{L}$-smooth for some $L\geq 2$. We denote individual mappings in the family by $\Pi_\lambda (\omega)=\Pi(\lambda, \omega)$ and refer to $\Lambda$ as the {\it parameter space} of the family. As is common in the field, we will refer to the mappings $\Pi_\lambda$ as \emph{projections} and to $\Pi$ as a \emph{projection family}.

\begin{definition}\label{defin_transversality_Euclidean}
Let $\Pi:\Lambda \times \Omega\rightarrow \R^m$ be a family of mappings that is $C^{L}$-smooth for some $L\geq 2$. For  $\lambda\in \Lambda{}$ and $v\neq  w \in \Omega$,  define
\begin{equation}\label{eq:definePhi}
\Phi(\lambda,v,w)=\frac{\Pi(\lambda,v)-\Pi(\lambda,w)}{|v-w|}.
\end{equation}

The family $\Pi$ is {\it transversal} on $\Lambda\times \Omega$ if
there exists  $C>0$ such that for all $\lambda\in \Lambda$ and all $v\neq  w\in \Omega$ we have:
 \begin{equation}\label{eq:transGeneral}
 \text{ if \ } |\Phi(\lambda,v, w)|\leq C \text{ \ then \ }
 \left|\det \Diff_\lambda\Phi(\lambda,v, w)(\Diff_\lambda\Phi(\lambda,v, w))^\mathsf{T}\right|\geq C^2,
\end{equation}
where $(\Diff_\lambda\Phi(\lambda,v, w))^\mathsf{T}$ denotes the transpose of the $(m\times k)$-matrix $\Diff_\lambda\Phi(\lambda,v, w)$. 

The family $\Pi$ is {\it locally transversal} if $\Lambda\times \Omega$ can be covered by neighborhoods such that the restriction of $\Pi$ to each of these neighborhoods is transversal.
\end{definition}

In the case $m=1$, the transversality condition \eqref{eq:transGeneral} can be rewritten as:
\begin{equation}\label{eq:trans_m1}
    \text{ if \ } |\Phi(\lambda,v, w)|\leq C \text{ \ then \ }
\max_{j\in\{1,\dots,k\}}  \left| \dm_{\lambda_j}\Phi(\lambda,v, w) \right|\geq C.
\end{equation}

For $m=k=1$, it further reduces to:
\begin{equation}\label{eq:trans_m1k1}
    \text{ if \ } |\Phi(\lambda,v, w)|\leq C \text{ \ then \ }
 \left| \dm_\lambda \Phi(\lambda,v, w)\right|\geq C.
\end{equation}

The basic intuition for the relation between projection theory and the transversality condition is as follows. Projection theorems fail for a family $\Pi(\lambda, \omega)$ if big parts of $\R^n$ are collapsed by many projections in the family. The if-side of \eqref{eq:transGeneral} detects this collapse for a given value of $\lambda$. The then-side then guarantees that two points $u,v$ that were collapsed under $\Pi_\lambda$ move apart quickly if the value of $\lambda$ is varied. \eqref{eq:trans_m1k1}.\\[-6pt]

The following theorem combines \cite[Theorem 4.9]{PS2000} and \cite[Theorem 1.2]{HJJL2012}.

\begin{theorem}\label{thm:PS_n} Let $n,m,k\in\N$ with $k\geq m$ and let $L\geq 2$.
Let $\Pi: \Lambda{}\times \Omega \rightarrow \R^m$ be a locally transversal, $C^{L}$-smooth family of mappings, then for all Borel sets $A\subset \Omega$: \begin{enumerate}[label={(\arabic*)\ }, itemsep=2pt, topsep=1pt]
\item If $\dim A \leq m$, then 
\begin{enumerate}[label={(\alph*)\ },itemsep=2pt,]
\item  $\dim (\Pi_\lambda A)= \dim A$ for $\Haus^k$-a.e.\ $\lambda\in  \Lambda{}$,
\item For $0<\alpha\leq\dim A$, $\dim(\{\lambda \in  \Lambda{}\ : \ \dim(\Pi_\lambda A)<\alpha\})\leq (n-m-1)m+\alpha$.
\end{enumerate}
\item If $\dim A > m$, then 
\begin{enumerate}[label={(\alph*)\ }, itemsep=2pt,]
\item  $\Haus^m (\Pi_\lambda A)>0$ for $\Haus^k$-a.e.\ $\lambda\in  \Lambda{}$,
\item $ \dim(\{\lambda \in  \Lambda{}\ : \ \mathscr{L}^m(\Pi_\lambda A)=0\})\leq (n-m)m+m-\min\{\dim A,L-1\}$.
\end{enumerate} 
\item If $\dim A>2m$, then
\begin{enumerate}[label={(\alph*)\ },itemsep=2pt,]
\item  $\Pi_\lambda A\subset \R^m$ has non-empty interior for $\Haus^k$-a.e.\ $\lambda\in  \Lambda{}$,
\item \begin{tabbing} $\dim(\{\lambda \in  \Lambda{} :  (\Pi_\lambda A)^\circ \neq \lm \})\leq (n-m)m-(\min\{\dim A,L-1\} -2m)(1+\frac{m}{L})^{-1}$
\end{tabbing}
\end{enumerate}

\item If $\Haus^m(A)<\infty$, then $A$ is purely $m$-unrectifiable if and only if $\Haus^m(\Pi_\lambda(A))=0$ for~$\Haus^k$-a.e.\ $\lambda\in  \Lambda{}$.

\end{enumerate}
\end{theorem}

\subsection{Preservation of transversality under coordinate changes}
\label{subsec:Auxiliary}
We now show that transversality property is  locally-preserved by pre- and post- composition with $C^2$-diffeomorphisms. This assumption aligns with our standing assumption that projection families are $C^2$-smooth.

\begin{lemma}
\label{lemma:transversalitycomposition}
Let $L\geq 2$ be an integer, and consider the following change of coordinates, where $\Lambda{},\tilde \Lambda{}\subset \R^k$, $\Omega,\tilde \Omega\subset \R^n$, and $U, \tilde U\subset \R^m$ are open domains, $f:\Lambda\rightarrow \tilde \Lambda$, $g:\Omega\rightarrow \tilde \Omega$, and $h:U\rightarrow \tilde U$ are $C^L$-diffeomorphisms, and $\tilde \Pi$ is given by $\tilde \Pi(f(\lambda),g(\omega)) =h(\Pi(\lambda,\omega))$.

\[\begin{tikzcd}
	{\Lambda \times \Omega} && U \\
	\\
	{\tilde \Lambda \times \tilde \Omega} && {\tilde U}
	\arrow["\Pi", from=1-1, to=1-3]
	\arrow["{\tilde \Pi}", from=3-1, to=3-3]
	\arrow["{f \times g}"{description}, from=1-1, to=3-1]
	\arrow["h"{description}, from=1-3, to=3-3]
\end{tikzcd}\]
If $\Pi$ is $C^L$-smooth and locally transversal, then so is $\tilde \Pi$. 
\end{lemma}

We will prove a quantitative version of the transversality assertion in Lemma \ref{lemma:transversalitycomposition} (the preservation of smoothness under composition is a standard fact):
\begin{lemma}\label{lemma:transversalitycompositionquantitative}
Under the assumptions and in the notation of Lemma \ref{lemma:transversalitycomposition}, let $(\lambda_0, \omega_0)\in \Lambda{}\times \Omega$. Then, there exists a neighborhood $\Lambda_0\times \Omega_0 \subset \Lambda\times \Omega $ of $(\lambda_0, \omega_0)$ and a constant $C_0>0$ such that: whenever $\Pi$ satisfies the transversality condition \eqref{eq:transGeneral} for some constant $C>0$ for a  triple $(\lambda,v, w)\in \Lambda_0\times \Omega_0^2$, then $\tilde \Pi$ satisfies \eqref{eq:transGeneral} with constant $CC_0>0$ for the triple $(f(\lambda),g(v),g(w))$.

Furthermore, $C_0$ depends only on the local bilipschitz constants of $f$, $g$, $h$ and bounds on $\Diff h$ and $\Diff f^{-1}$. Furthermore, $C_0$ approaches $1$ as $f$ and $g$ approach the identity and $h$ approaches a linear mapping, in terms of these constants.
\end{lemma}


For a matrix $M$, let $\rho(M)=\det[MM^\mathsf{T}]^{1/2}$. Note that this determinant is always non-negative (cf.~the Gram determinant). We will need the following composition result:

\begin{lemma}\label{lem:rholemma}
There is a continuous positive function $\sigma: GL(k,\R)\rightarrow \R$ satisfying $\sigma(\id)=1$ such that for all $A\in M^{m\times k}$  and  $B\in GL(k,\R)$one has
$\rho(AB)\geq \rho(A)\sigma(B)$.
\end{lemma}
\begin{proof}
If $k<m$, then the statement is trivial, as both sides of the equality are zero.

Consider first the function $\tau: GL(k,\R)\times G(k,m)\rightarrow \R$ which maps a pair $(M,V)$ to the volume distortion of the mapping $M\vert_V: V \rightarrow M(V)$. The function $\tau$ varies smoothly with $M$ and $V$ (indeed, it can be written down explicitly using Gram matrices and an identification of $V$ with $\R^k$). Noting that $G(k,m)$ is compact, set $\sigma(B) = \min_{V\in G(k,m)}\tau(B^\mathsf{T},V)^2$. Clearly, $\sigma$ is continuous and $\sigma(\id)=1$.

Consider now an arbitrary pair $A,B$ as in the statement of the lemma, and write $a=A^\mathsf{T}$ and $b=B^\mathsf{T}$. We then need to show that $\det((ab)^\mathsf{T}(ab))\geq \det(a^\mathsf{T}a)^{1/2}\sigma(B)$. The lemma follows immediately from the fact that $\det(a^\mathsf{T} a)^{1/2}$ computes the $m$-volume distortion induced by $a:\R^m\rightarrow \R^k$ and $\det((ab)^\mathsf{T} (ab))^{1/2}$ computes the $m$-volume distortion induced by $ab:\R^m\rightarrow \R^k$.
\end{proof}

\begin{proof}[Proof of Lemma \ref{lemma:transversalitycompositionquantitative}] Let $(\lambda_0,\omega_0)\in \Lambda\times \Omega$ and let $\Lambda_0\times \Omega_0 \subset \Lambda\times \Omega $ be a neighborhood of $(\lambda_0, \omega_0)$ that is sufficiently small so that the specific Lipschitz and co-Lipschitz constants that we will explicitly define in the sequel of this proof exist.
Furthermore, we write $U_0=\Pi(\Lambda_0\times \Omega_0^2)$ and $\Psi(\lambda, v,  w):=\Pi(\lambda, v)-\Pi(\lambda,  w)$ as well as $\tilde \Psi(\tilde \lambda, \tilde v, \tilde w):= \tilde \Pi(\tilde\lambda, \tilde v)-\tilde \Pi(\tilde\lambda, \tilde  w)$ for points $(\lambda, v, w)\in \Lambda\times(\Omega)^2$ resp.\ $(\tilde\lambda,  \tilde v,\tilde w)\in \tilde\Lambda\times(\tilde\Omega)^2$.

Now, let $(\lambda,v, w)\in \lambda_0\times \Omega_0^2$ be a fixed triple and set $(\tilde \lambda, \tilde v, \tilde  w)=(f(\lambda),g(v),g(w)) $. Let $C>0$ a constant. We assume that the transversality condition \eqref{eq:transGeneral} for the family $\Pi$ holds for  $(\lambda,v, w)$ with constant $C$, that is,
\begin{equation}\label{eq:complemma_assume_transv}
 \text{ if \ } \frac{|\Psi(\lambda,v, w)|}{|v- w|}\leq C \text{ \ then \ }
\frac{ \left|\det[ \Diff_\lambda\Psi(\lambda,v, w)(\Diff_\lambda\Psi(\lambda,v, w))^\mathsf{T}]\right|}{|v- w|^2}\geq C^2.
\end{equation}

We analyze pre-composition (with $f$ and $g$) and post-composition (with $h$) separately.

For \underline{pre-composition}, we assume $h$ is the identity.
 Let $ C_1>0$ be the co-Lipschitz constant of $g$ on $\Omega_0$ and assume that $\tilde \Pi$ satisfies the if-part of the transversality condition \eqref{eq:transGeneral} with constant $CC_1$   for the triple $(\tilde \lambda, \tilde v, \tilde  w)$, that is, \begin{equation}
 \label{eq:precompLHS}
 \frac{| \tilde \Psi(\tilde \lambda, \tilde v, \tilde  w)|}      {\norm{\tilde v-\tilde  w}}\leq CC_1.\end{equation} 

 By definition of $C_1$ and since $\Psi(\lambda,u,v)=\tilde \Psi(\tilde \lambda, \tilde v, \tilde  w)$ by definition, equation \eqref{eq:precompLHS} implies the if-part of equation \eqref{eq:complemma_assume_transv}. Hence also the then-part of equation \eqref{eq:complemma_assume_transv} follows.
 
  Notice that since $\tilde{\Psi}(\tilde \lambda, \tilde  v, \tilde  w )= \Psi(f^{-1}(\tilde \lambda), g ^{-1}(\tilde  v)), g^{-1} (\tilde  w))$, by chain rule for differentiation we can write 
 \begin{equation}\label{eq:precomp_chainrule}
     \Diff_{\tilde \lambda}\tilde{\Psi}(\tilde \lambda, \tilde  v, \tilde  w ) =\Diff_\lambda \Psi(f^{-1}(\tilde \lambda), g ^{-1}(\tilde  v)), g^{-1} (\tilde  w))\cdot \Diff f^{-1}(\tilde \lambda).
 \end{equation}
  Let $C_2:= \min\{\sigma(\Diff f^{-1}(\tilde \lambda)): \tilde \lambda \in f(\Lambda_0)\}$ where $\sigma$ is as in Lemma~\ref{lem:rholemma}. Then
 \begin{equation}
 \label{eq:precompRHS_2}
 \frac{|\det\left [\Diff_{\tilde\lambda} \tilde \Psi( \tilde\lambda, \tilde v, \tilde w)\Diff_{\tilde\lambda} \tilde \Psi( \tilde\lambda, \tilde v, \tilde w)^\mathsf{T}\right]|}      {\norm{\tilde v-\tilde w}^{2}}\geq C^2 (C_1C_2)^2.
 \end{equation}
Choosing $C_0:=\min\{C_1,C_1C_2\}$ concludes the proof for precomposition.

For \underline{post-composition}, we assume that $f$ and $g$ are the respective identity mappings.
 Let $C_3$ be the co-Lipschitz constant of $h$ on $U_0$. Recall that here $\tilde \Psi( \lambda,   v,   w )= h(\Pi( \lambda,  v))-h(\Pi( \lambda,  w))$ and assume that \begin{equation}
 \label{eq:postcomp_LHS}
 \frac{| h(\Pi(\lambda,  v))-h(\Pi(\lambda,  w))|}      {\norm{  v-  w}}\leq CC_3K \end{equation} where $0<K\leq 1$ is a constant that will be explicitely defined later. ($K$ will be independent of the constant $C$ as well as independent of the choice of the points $\lambda,v,w$.) By equation \eqref{eq:postcomp_LHS} and the choice of $C_3$ it follows that
\begin{equation}\label{eq:postcomp_LHS2}
\frac{| \Pi(\lambda,  v)-\Pi(\lambda,  w)|}      {\norm{  v-  w}}\leq KC\leq C \end{equation} and by \eqref{eq:complemma_assume_transv}
 \begin{equation}\label{eq:postcomp_RHS1}
 \frac{|\det\left [(\Diff_\lambda \Psi( \lambda,  v,  w))(\Diff_\lambda \Psi( \lambda,  v,  w))^\mathsf{T}\right]|}      {\norm{ v- w}^2}\geq C^2.\end{equation}
 Hence by the chain rule for differentiation:
 \begin{align}\label{eq:postcomp_chainrule1}
     \Diff_{\lambda}\tilde{\Psi}(\lambda,  v,  w )
     &=\Diff h(\Pi(\lambda, v))\cdot 
     \Diff_\lambda \Pi(\lambda,  v)-\Diff h(\Pi(\lambda, w))\cdot 
     \Diff_\lambda \Pi(\lambda,  w)\\ \label{eq:postcomp_chainrule2}
     & = \Diff h(\Pi(\lambda, v))\cdot \left[
     \Diff_\lambda \Psi(\lambda,  v, w)\right]\\
     & \ \ \ \label{eq:postcomp_chainrule3}
     +\left[ \Diff h(\Pi(\lambda, v)) - \Diff h(\Pi(\lambda, w)) \right]\cdot \Diff_\lambda \Pi(\lambda, w)
 \end{align}
 To abbreviate the notation in the following computations, we denote the product of matrices on line \eqref{eq:postcomp_chainrule2} by $A$ and the product of matrices on line \eqref{eq:postcomp_chainrule3} by $B$. So, in particular $\Diff_{\lambda}\tilde{\Psi}(\lambda,  v,  w )=A+B$. We will now show that the determinant of $A$ is big and that $B$ is small in a suitable sense. To this end, let $C_4:=\min\{|\det \Diff h(u)|: u\in U_0\} $.
 Since $\Diff h(u)$ for $u\in U$ is a square matrix and by \eqref{eq:postcomp_RHS1}, we have
 \begin{equation}\label{eq:estimateA}
     \frac{\det[AA^\mathsf{T}]}{\norm{ v- w}^2}= \det\left[\Diff h(\Pi(\lambda, v))\right]^2\frac{\det \left[
     \Diff_\lambda \Psi(\lambda,  v, w)\Diff_\lambda \Psi(\lambda,  v, w)^\mathsf{T}\right]}{\norm{ v- w}^2}\geq C_4^2C^2.
 \end{equation} 
 Notice that in the case where $h$ is a linear mapping, then we have $B=0$ and thus \eqref{eq:estimateA} concludes the proof for postcomposition (choose $K=1$ and $C_0=\min\{C_3,C_4\}$). However, we may in general not assume that $h$ is linear. 
 Let $C_5\geq 0$ be the local Lipschitz constant of the mapping $\Diff h: U_0 \to \R^{m\times m}$  in terms of the Euclidean norm 
$|\cdot|$ on $U_0$ and the supremum norm $\|\cdot\|_\infty$ (i.e. absolute value of maximal entry)  on the space $\R^{m\times m}$ of square matrices. (Notice that $C_5=0$ is the case where $h$ is linear). By definition of $C_5$ and by \eqref{eq:postcomp_LHS2} we have
  \begin{equation}\label{eq:estimateB1}
  \frac{\|\Diff h(\Pi(\lambda, v))-\Diff h(\Pi(\lambda, w))\|_\infty}{\norm{ v- w}}\leq C_5 \frac{|\Pi(\lambda, v)-\Pi(\lambda, w)|}{\norm{ v- w}} \leq C_5 C K.
  \end{equation}
  It is an easy to check fact that the supremum norm for matrices has the following property: For matrices $M\in \R^{m\times m}$ and  $N\in \R^{m\times l}$ (where $l\in N$), $\|MN\|_\infty\leq m\|M\|_\infty\|N\|_\infty$. Define $C_6:=\sup\{\|\Diff_\lambda \Pi(\lambda,\omega)\|_\infty : \omega\in\Omega_0\} $. Combining this fact with equation \eqref{eq:estimateB1} yields 
  
  \begin{equation}
 \label{eq:estimateB2}
  \frac{||B||_\infty}{{\norm{ v- w}}}\leq m  \frac{\|\Diff h(\Pi(\lambda, v))-\Diff h(\Pi(\lambda, w))\|_\infty}{\norm{ v- w}} \|\Diff_\lambda \Pi_2\|_\infty\leq mC_5KC_6C.
 \end{equation}
 
 Now consider the function $\rho: \R^{m\times k}\to \R$ defined by $\rho(M)=\det [MM^\mathsf{T}]^\frac{1}{2}$. We equip $\R^{m\times k}$ with the supremum norm and $\R$ with the Euclidean norm. Then $\rho$ is continuous on $\R^{m\times k}$ and it is Lipschitz on compacta in  $\R^{m\times k}\setminus \rho^{-1}(\{0\})$. 
 By continuity of $\rho$ and by choosing $K$ sufficiently small, employing equations \eqref{eq:estimateA} and \eqref{eq:estimateB2}, it follows that for all point triples $\lambda\in \lambda_0,v,w\in \Omega_0 $ , the matrices $A$ and $A+B$ live in a compactum in $\R^{m\times k}\setminus \rho^{-1}(\{0\})$. Set $C_7>0$ to be the Lipschitz constant of $\rho$ on said compactum with respect to the metric $\|\cdot\|_\infty$ on $\R^{m\times k}$ and $|\cdot|$ on $\R$.
 Hence, 
 \begin{equation}\label{eq:estimateA+B}
    | \det[(A+B)(A+B)^\mathsf{T} ]^\frac{1}{2}- \det[AA^\mathsf{T} ]^\frac{1}{2}|\leq C_7\|(A+B)-A\|_\infty\leq \|B\|_\infty
 \end{equation}
 and finally
 \begin{equation}\label{eq:estimateA+Bfinal}
 \frac{\det[\Diff_\lambda\tilde\Psi(\Diff_\lambda\tilde\Psi)^\mathsf{T}] ^\frac{1}{2}}{\norm{v-w}} 
 \geq  \frac{\det[AA^\mathsf{T}]^\frac{1}{2}}{\norm{v-w}}-C_7 \frac{\|BB^\mathsf{T} \|_\infty}{\norm{v-w}}\geq (C_4-mKC_5C_6C_7)C.
 \end{equation}
 We choose $0<K<\frac{C_4}{mC_5C_6C_7}$ (resp.\ $K=1$ in the case when $C_5=0$, see below equation \eqref{eq:estimateA}) and set $C_0:=\min\{KC_3, C_4-mKC_5C_6C_7\}$. Then the assumption \eqref{eq:postcomp_LHS} together with the conclusion in equation \eqref{eq:estimateA+Bfinal} conclude the proof for postcomposition. 
  \end{proof}
 
 \begin{remark} \emph{}
\begin{enumerate}[label={(\roman*)\ }, itemsep=2pt,topsep=1pt]
    \item A slightly more general form of Lemma \ref{lemma:transversalitycomposition} is  true with the same proof. Namely, we may assume that the parameter variable $\lambda$ depends on both the new $\tilde \lambda$ and $\tilde \omega$ variables. However, making the parameter $\lambda$ depend on the choice of the point $\tilde \omega$ does not seem like a natural scenario in the setting of projections. Moreover, notice that we cannot generally allow the space variable $\omega$ to depend on both $\tilde \lambda$ and $\tilde \omega$. Lemma~\ref{lemma:transversalitycomposition} may indeed fail in that case.
    \item In \cite{PS2000} and \cite{Mattila2004} the regularity assumptions are much weaker: they assume the differentiability of $\Pi(\lambda,\omega)$ in $\lambda$ and the boundedness of (all orders of) derivatives $\Diff_\lambda$ (locally) uniformly in $\Omega$. In our smooth geometric settings these technicalities can be avoided. 
    If we disregard differentiability of $\Pi$ on the product space and are only interested in the preservation of transversality in Lemma \ref{lemma:transversalitycompositionquantitative}, then it suffices to assume $C^1$-smoothness for the diffeomorphisms $f,g,h$ and the existence of the constants $C_2, C_4,$ and $C_5$ in the proof, defined in terms of the differentials of $f$ and $h$.
    This makes Lemma \ref{lemma:transversalitycompositionquantitative} applicable in more general settings where much weaker regularity is assumed for projection families.
\end{enumerate}
\end{remark}

\subsection{Local transversality for projection families on manifolds}\label{sec:trans_on_mfds}

 In view of Lemma \ref{lemma:transversalitycomposition}, we may speak of local transversality for a projection family whose domain, parameter space, and/or target are smooth manifolds, by working in coordinates.

Let $n,m,k\in \N$ with $k\geq m$.  Let $N$, $M$, and $\Lambda$ be manifolds of dimension $n$, $m$, and $k$ respectively.

As in the Euclidean case (beginning of Section \ref{subsec:PrelimTransversality}) we refer to a continuous mapping $\Pi:\Lambda\times N\to M$ as a \emph{family of projections with parameter space $\Lambda$} and we denote individual mappings in the family by $\Pi_\lambda (p)=\Pi(\lambda, p)$.

\begin{definition}
Let $L\geq 2$. We call $\Pi$ a {\it $C^L$-smooth projection family} if the following regularity assumptions hold:
$\Lambda$, $N$, and $M$ are each equipped with a $C^L$-smooth atlas and $\Pi$ is $C^{L}$-smooth.\end{definition}

\begin{definition}
A $C^L$-smooth projection family $\Pi$ is called {\it locally transversal} if Definition~\ref{defin_transversality_Euclidean} is satisfied in terms of the $C^L$-smooth local coordinates of $\Lambda$, $N$, and $M$.
\end{definition}

\subsection{Transversality for families induced by group actions}\label{sec:trans_for_groups}
Recall that an action of a Lie group $\Gamma$ on a manifold $N$ is smooth if the corresponding mapping $\Gamma\times N \rightarrow N$ is smooth.

We first show that transversality is easier to check for families induced by group actions:

\begin{lemma}
\label{lemma:thetaZero}
Let $N$ be a smooth manifold, $\Gamma$ a Lie group acting smoothly on $N$, and $S_0\subset N$ a closed subset. Let $S=\{(\gamma,p)\in \Gamma\times N: \gamma(p)\in S_0 \}$. For a projection family $\Pi:\Gamma\times N \rightarrow M$ with domain $(\Gamma\times N)\setminus S$ induced by the action of $\Gamma$ on $N$, the following are equivalent:
\begin{enumerate}[label={(\roman*)\ }, itemsep=2pt,topsep=1pt]
    \item (Local transversality) the family $\Pi$ is locally transversal.
    \item (Local transversality near the identity) for every point $(\gamma,v)$ in the domain with $\gamma=\id$, there exists a neighborhood in $\Gamma\times N$ so that $\Pi$ restricted to that neighborhood is transversal.
    \item (Local transversality at the identity) there exists a neighborhood $U$ in $(N\setminus S_0)$ and a constant $C>0$ such that the transversality condition \eqref{eq:transGeneral} holds for all triples $(\id,  v,  w)$ in the domain satisfying $v,  w \in U$. (By the requirement that triples $(\id,  v,  w)$ should lie in the domain, we mean that $(\id,v),(\id,w)\in (\Gamma\times N)\setminus S$).
\end{enumerate}
\end{lemma}
\begin{proof}
It is immediate from the definition of local transversality that (1) implies (2) and that (2) implies (3).

To prove that (2) implies (1), suppose that $(\gamma_0,v_0)\in (\Gamma\times N)\setminus S$. By the assumption of~(2), the projection family $\Pi$ is transversal in a neighborhood of $(\id, \gamma_0 v_0)$. By Lemma~\ref{lemma:transversalitycomposition}, the projection family $\tilde \Pi(\gamma,v):=\Pi(\gamma \gamma_0^{-1},\gamma_0 v)$ is transversal in a neighborhood of $(\gamma_0,v_0)$. Observe that by the definition of projection families induced by group actions, we have $\Pi(\gamma \gamma_0^{-1},\gamma_0 v)=\pi((\gamma \gamma_0^{-1})\gamma_0 v) = \pi(\gamma v )=\Pi(\gamma, v) $. Hence, the family $\Pi$ is transversal in a neighborhood of $(\gamma_0,v_0)$.

We now show that (3) implies (2). Suppose we are interested in checking the transversality condition \eqref{eq:transGeneral} for a triple $(\gamma_0,  v,  w)$. By Lemma \ref{lemma:transversalitycompositionquantitative}, it suffices to check transversality 
for the family $\tilde \Pi(\gamma,v):=\Pi(\gamma \gamma_0^{-1},\gamma_0 v)$ at the corresponding triple $(\id, \gamma_0 v, \gamma_0 w)$.  Furthermore, since the distortion of the transversality condition in Lemma \ref{lemma:transversalitycompositionquantitative} vanishes when the variable $\gamma$ approaches the identity, transversality of $\Pi$ for triples of the form $(\id,  v,  w)$ in fact implies transversality of $\Pi$ for triples of the form $(\gamma,  v,  w)$ ($v,w\in U$) for $\gamma$ in neighborhood $V$ of $\id$ in $\Gamma$. Hence $V\times U\subset \Gamma\times N$ is the neighborhood for (2).
\end{proof}

Lastly, we show that, for projection families induced by group actions, it suffices to check transversality on a smaller family that is induced by the action of a closed subgroup.
\begin{prop}
\label{prop:enlargedFamily}
Let $G$ be a Lie group acting on a manifold $N$, $\Gamma\subset G$ a closed Lie subgroup of $G$, $M$ a manifold,
and $\pi: N\rightarrow M$ a smooth projection. If the projection family $\Pi: \Gamma\times N\rightarrow M$ given by $\Pi(\gamma,p)=\pi(\gamma (p))$ is locally transversal, then so is the projection family $\tilde \Pi: G \times N \rightarrow M$ given by $\tilde \Pi(g,p)=\pi(g (p))$.
\end{prop}
\begin{proof}
We first prove the theorem in the case of a one-dimensional target (which is the only case we will use). 

By Lemma \ref{lemma:thetaZero}, it suffices to prove transversality at the identity. That is, we need to show that for some $C>0$:
 \begin{equation}
 \label{eq:gTransversality}
 \text{ if \ }
 \frac{| \pi(  v) - \pi( w)|}
      {\norm{ v- w}}\leq C,
 \text{ \  then \  }\left|\dm_g\vert_{g=\id}
 \frac{\pi( g v) - \pi(g w)}
      {\norm{ v- w}}
 \right|\geq C.
\end{equation}

Now, since $\Gamma\subset G$ is a closed Lie subgroup, by the Homogeneous Space Construction Theorem and Quotient Manifold Theorem  \cite[Theorems 21.17 and 21.10]{Lee2013} we have that the quotient space $H=G/\Gamma$ (consisting of $\Gamma$-cosets of $G$) is a manifold; indeed, around any point of $G$ there is a chart on $G$ with coordinates $(\vec x, \vec y)$ such that $\vec x$ represents points in $\Gamma$ and $\vec y$ represents points in $H$. Using such a coordinate chart at the identity of $G$, we see $G$ locally as a product of $\Gamma$ and $H$; in particular, taking $0_H$ to be the $\Gamma$-coset passing through the identity of $G$, we can write $\id_G=(\id_\Gamma, 0_H)$. We can then vary $g$ either along $\Gamma$ or along $H$, so that $\dm_g\vert_{g=\id}$ can be decomposed as $(d\gamma\vert_{\gamma=\id_\Gamma}, dh\vert_{h=0_H})$. By transversality of the $\Gamma$-based projection family $\Pi$, we know that $\left|\dm_\gamma\vert_{\gamma=\id_\Gamma}
 \frac{\pi( \gamma v) - \pi(\gamma w)}
      {\norm{ v- w}}
 \right|\geq C
$, which then immediately gives the then-side of \eqref{eq:gTransversality}, completing the proof for one-dimensional targets.

For higher-dimensional targets, one uses the more general Definition \ref{defin_transversality_Euclidean} of transversality, and end of the proof relies on the following lemma from linear algebra.
\end{proof}

\begin{lemma}
Let $M=(A\; B)$ be a real block matrix. Then $\det MM^t \geq \det AA^t$.
\end{lemma}
\begin{proof}
Let $\alpha = A^t:\R^c \rightarrow \R^a$, $\beta=B^t:\R^c \rightarrow \R^b$, and $\mu =\twovector{\alpha}{\beta}=M^t:\R^c \rightarrow \R^{(a+b)}$. We need to show $\det \mu^t \mu \geq \det \alpha^t \alpha$.  It is a standard fact\footnote{For square matrices, this is can be seen by decomposing $T$ into a product of elementary matrices; while for non-square matrices this can be seen by an appropriate change of coordinates. 
}
that for any matrix $T:\R^c\rightarrow \R^d$, the quantity $\det T^t T$ is the distortion of the $c$-dimensional volume element. Thus, we are trying to show that for any Borel set $A\subset \R^c$, $\Haus_c(\mu(A))\geq \Haus_c(\alpha(A))$. This follows immediately from the fact that $\alpha(A)=\pi(\mu(A))$ for the standard projection $\pi: \R^{(a+b)}\rightarrow \R^a$.
\end{proof}

\begin{remark}
The lemma fails over $\C$ if we take $A=1$ and $B=\ii$, so that $AA^t+BB^t=0$. This explains the need for a geometric argument.
\end{remark}

To illustrate Proposition \ref{prop:enlargedFamily}, consider the following simple case:
\begin{example}
Let $N=\R^2$ be Euclidean space, and $\pi:\R^2\rightarrow \R$ given by $\pi(x,y)=x$. It is well-known that the family of projections associated with the group $\Gamma=O(2)$ of rotations around the origin is transversal. One can enlarge $\Gamma$ to several reasonable choices of $G$, such as the isometry group $\operatorname{Isom}(\R^2)$ which adds translations, the similarity group $\operatorname{Sim}(\R^2)$ which adds aspect-preserving dilations, the linear group $GL(2,\R)$, or the group of affine motions $\operatorname{Aff}(\R^2)$ that includes all above motions. By Proposition \ref{prop:enlargedFamily}, each of these groups will give rise to transversal families of projections.
\end{example}

\section{M\"obius transformations}
\label{sec:season1}

In this section, we prove Theorems \ref{thm:MobAllProjections} and \ref{thm:MobSpecialProjections} concerning families of projections induced by subgroups of $\Mob$.  We then apply Theorem \ref{thm:MobSpecialProjections} to specific projection families.  Theorem \ref{thm:MobAllProjections} follows immediately from Proposition \ref{prop:enlargedFamily} and either Theorem \ref{thm:MobSpecialProjections} or the well-known transversality of the family of orthogonal projections onto lines in $\R^2$.

To prove  Theorem \ref{thm:MobSpecialProjections}, we will analyze transversality more carefully for an arbitrary one-dimensional subgroup $\Gamma\subset \Mob$. We will think of $\Gamma$ as the image under the exponential map of a line in the Lie algebra of $\Mob$. To this end, we identify $\Mob$ with $SL(2,\C)$ and its Lie algebra with the algebra $\mathfrak{sl}(2,\C)$ of traceless 2-by-2 complex matrices.

\begin{theorem}
\label{thm:MobTransversality}
Let $\Gamma\subset \Mob$ be an arbitrary one-dimensional Lie subgroup, parametrized as $\Gamma=\{\gamma_t=\exp(At):t\in\R\}$ for some non-zero $A=(a_{ij})\in \mathfrak{sl}(2,\C)$. Define:
\begin{itemize}[itemsep=2pt,topsep=1pt]
\item \ $S_0=\{\infty\}\subset \hat \C$ and $S=\{(\gamma, z): \gamma(z)=\infty\}\subset \Gamma \times \hat \C$.

\item \ $L_0=\{z:\Im(a_{11}-a_{21}z) = 0\}\subset \C$ and $L=\{(\gamma, z):\gamma(z)\in L_0\}\subset \Gamma \times \hat \C$.

\item \ $K_0= L_0\cup S_0\subset \hat \C$ and $K= L\cup S\subset \Gamma\times \hat \C$.
\end{itemize}
Then the projection family $\Pi:\Gamma\times \hat \C\rightarrow \R$ given by $(\gamma, z)\mapsto \Re(\gamma(z))$ is defined on the domain $(\Gamma\times \hat \C)\setminus S$ and is locally transversal on $(\Gamma\times \hat \C)\setminus K$.
\end{theorem}

\noindent For a geometric description of the sets $L_0$ and $K_0$, see the proof of Theorem \ref{thm:MobSpecialProjections} below.

\begin{proof}
Linearizing at $t=0$, we have $\gamma_t=\id+At+O(t^2)=\begin{pmatrix}
1+a_{11}t & a_{12}t \\
a_{21}t & 1-a_{11}t
\end{pmatrix}+O(t^2)$ and 
\begin{equation}
    \Pi(t,z)= \Re\left(\frac{(1+a_{11}t)z+a_{12}t+O(t^2)}{a_{21}tz+(1-a_{11}t)+O(t^2)}\right).
\end{equation}
Noting that $\Re$ commutes with differentiation with respect to the real variable $t$, we have
\begin{equation}\label{eq:PiasRe}
    \dm_t \vert_{t=0} \Pi(t,z) = \Re \left(a_{12} +2 a_{11}z - a_{21}z^2 \right).
\end{equation}

In order to show locally transversality away from $K$, by Lemma \ref{lemma:thetaZero}, it suffices to check the transversality condition \eqref{eq:trans_m1k1} at $t=0$ locally for $(v,w)$, that is: around every $(v_0,w_0)\in \C^2$ with $v_0,w_0\notin K_0$ there exists a neighborhood $U\subset\C^2$ and a constant $C>0$ such that  
\begin{equation}\label{eq:transvinproof}
    \text{ if }\frac{\norm{\Psi(0,v,w)}}{\norm{v-w}}<C \text{ then }\frac{\norm{\dm_t \vert_{t=0} \Psi(t,v,w)}}{\norm{v-w}}>C
\end{equation}
for all $v,w\in U$, where for $\Psi(t,v,w):=\Pi(t,v)-\Pi(t,w)$.

Let $(v_0,w_0)\in \C^2$, $U\subset \C^2$ an open neighborhood around $(v_0,w_0)$, $C>0$ a small positive constant, and $(v,w)\in U$.
We write $v$ and $w$ in terms of their real and imaginary parts: $v=(x+\Delta x)+(y+\Delta y)\ii$ and $w=x+y\ii$, and note that $\Psi(0,v,w)=\Re(v-w)=\Delta x$.
Assume that the if-side of \eqref{eq:transvinproof} holds, that is, assume that \begin{equation}
    \label{eq:deltax}\frac{\norm{\Delta x}}{\norm{v-w}}<C.
\end{equation} Let $C_1>0$ such that  $C_1^{-1}(|\Delta x|+|\Delta y|) \leq|v-w|\leq C_1 (|\Delta x|+|\Delta y|)$ for all $(v,w)\in U$ (i.e. $C_1$ is the local bi-Lipschitz constant for Euclidean norm and the $1$-norm). Hence it follows that \begin{equation} \label{eq:deltay}\frac{\Delta y}{|v-w|} \geq \frac{1-C}{C_1}. \end{equation}

First applying \eqref{eq:PiasRe} and then writing out all parts in terms of $x,y,\Delta x, \Delta y$, we have
\begin{align}
  \label{pretermdots}  \frac{\norm{\dm_t \vert_{t=0} \Psi(t,v,w)}}{\norm{v-w}} &=\norm{\Re \frac{(a_{12} +2 a_{11}v - a_{21}v^2) - (a_{12} +2 a_{11}w - a_{21}w^2)}{\norm{v-w}}}\\
    &=\norm{\Re \frac{2 a_{11}(v-w)  + a_{21}(-v^2 + w^2)}{\norm{v-w}}}\\ \label{termdots}
    & = \norm{\Re \left( p(x,y,\Delta x, \Delta y) \right)\frac{\Delta x}{\norm{v-w}}  + \Re ( 2 a_{11}\ii  - a_{21}(2 w\ii - \Delta y))\frac{\Delta y  }{\norm{v-w}}}.
\end{align}
where $p(x,y,\Delta x, \Delta y)$ is a polynomial in the variables $x,y,\Delta x$ and $\Delta y$. Hence the absolute value of $p$ is bounded from above by a finite constant $C_2>0$ depending on the neighborhood $U$. Combining this with \eqref{eq:deltax} and \eqref{eq:deltay} yields
\begin{align}\label{eq:preestimate}  \frac{\norm{\dm_t \vert_{t=0} \Psi(t,v,w)}}{\norm{v-w}}&\geq \frac{1-C}{C_1}\left|\Re( 2 a_{11}\ii  - a_{21}(2 w\ii)) - \Delta y)\right| -C_2C\\ \label{eq:estimate} &\geq
\frac{1-C}{C_1}\left|\Re( 2 a_{11}\ii  - a_{21}(2 w\ii)) )\right|- \frac{1-C}{C_1}\diam(U) -C_2C. \end{align}
Assuming that $v_0,w_0\notin S_0$ and that  $\diam U$ and $C$ are both  sufficiently small, we have
\[\frac{\norm{\dm_t \vert_{t=0} \Psi(t,v,w)}}{\norm{v-w}}\geq C\]
for all $(v,w)\in U$, as desired.

Now it remains to prove that transversality fails on $K$. It clearly suffices to consider a finite point $w_0\in K_0$. Notice that all computations and estimates up until equation \eqref{eq:estimate} are still valid under this assumption on $w_0$, and choose $w=w_0$, $v=w_0+\Delta y \ii$ (i.e.\ $\Delta x=0$). By the definition of $K_0$ and \eqref{termdots} it follows that
\[\frac{\norm{\dm_t \vert_{t=0} \Psi(t,v,w)}}{\norm{v-w}}=
\norm{\frac{ a_{21}(\Delta y)^2 }{\norm{v-w}}}\leq |a_{21}||\Delta y|,\]
which for small $\Delta y$ cannot be bounded away from zero and transversality fails, unless $a_{21}=0$. In this remaining case, we either have that $\Im(a_{11})\neq 0$ in which case $L_0=\emptyset$, or $\Im(a_{11})=0$, in which case $\Gamma$ acts by translations and transversality clearly fails on $L_0=\C$.
\end{proof}

We can now derive Theorem \ref{thm:MobSpecialProjections} from Theorem \ref{thm:MobTransversality}:
\begin{proof}[Proof of Theorem \ref{thm:MobSpecialProjections}]
Combining Theorem \ref{thm:MobTransversality} and Theorem \ref{thm:PS_n}, projection theorems hold wherever we have $\Im(a_{11}-a_{21}z)\neq 0$, where $A=(a_{ij})\in  \mathfrak{sl}(2,\C)$ is the Lie algebra of~$\Gamma$. It remains to describe this set geometrically.

If $a_{21}=0$, then $\Gamma$ preserves $\infty$. If, furthermore, $\Im(a_{11})\neq 0$, $\Im(a_{11}-a_{21}z)\neq 0$ holds for all $z$ and we obtain transversality on all of $\Gamma\times \C$. Conversely, if $\Im(a_{11})=0$, then transversality fails everywhere; exponentiating the matrix explicitly, one sees that the group is, in fact, either a translation $z\mapsto z+a_{12}t$ if $a_{11}=0$ or a dilation centered at $-a_{12}/(2a_{11})$ if $a_{11}\neq0$. So projection theorems fail everywhere.

If $a_{21}\neq 0$, then the set $L_0=\{z:\Im(a_{11}-a_{21}z)=0\}$ is a line, and one shows using the linearization of $\gamma_t$ that $L_0$ is tangent to the infinite orbit $\Gamma(\infty)$ at $t=0$, see Figure \ref{fig:infiniteTangent}.

In most cases, projection theorems can be recovered despite the lack of transversality along $L_0$. There are three cases:
\begin{enumerate}[label={(\arabic*)\ }, itemsep=2pt,topsep=1pt]
    \item (bad case) $L_0$ is a vertical line and $K_0=\Gamma(\infty)$. In this case, projection theorems fail for subsets of $L_0$ since they are inevitably projected to a single point, but hold elsewhere.
    \item (good case) $L_0$ is a non-vertical line and $K_0=\Gamma(\infty)$. In this case, the restriction of the projection to $L_0$ is a similarity mapping, so Hausdorff dimension and positivity of the Hausdorff measure are preserved by the projection $\pi$ along $L_0$.
    \item (artifact case) $K_0\neq \Gamma(\infty)$. In this case, any  set of sufficiently small diameter inside $L_0$ will be moved away from $L_0$ by $\gamma_t$ after some time. Indeed, the orbits of $\Gamma$ are analytic away from their endpoints and therefore cannot overlap $K_0$ in a relatively-open set without forcing $K_0$ to be an orbit. Projection theorems following from transversality then apply.
\end{enumerate}
This completes the proof of Theorem \ref{thm:MobTransversality}.
\end{proof}

We finish the section with some examples of M\"obius projection families.

\begin{example}
Classical projection theory focuses on the group ${O(2)=\{z\mapsto e^{\ii t}z:t\in\R\}}$ corresponding to Lie algebra element $A=\begin{pmatrix}
\ii/2 & 0 \\
0 & -\ii/2
\end{pmatrix}$. We recover the well-known transversality result for this family. 
\end{example}

\begin{example}
Compact one-parameter families $\Gamma$ of $\Mob$ are conjugate to $O(2)$ by some $M\in SL(2,\C)$ (one sees this by putting the generator $A$ of $\Gamma$ in Jordan canonical form). The group $\Gamma$ then fixes two points, $M(0)$ and $M(\infty)$. If one of these is infinite, then $\Gamma$ simply rotates around the other point. If both are finite, we may use a translation and dilation to normalize the two points to lie on the unit circle, see the left side of Figure \ref{fig:twoNormalizations}. Transversality fails along the linear orbit of $\infty$, but projection theorems are recovered as long as the orbit is not vertical. Conjugating the projection family by $M$ yields a ``circular point-source projection'' shown on the right in Figure \ref{fig:twoNormalizations}.
\end{example}

\begin{figure}
    \centering
    \hfill{}
    \includegraphics[width=.35\textwidth]{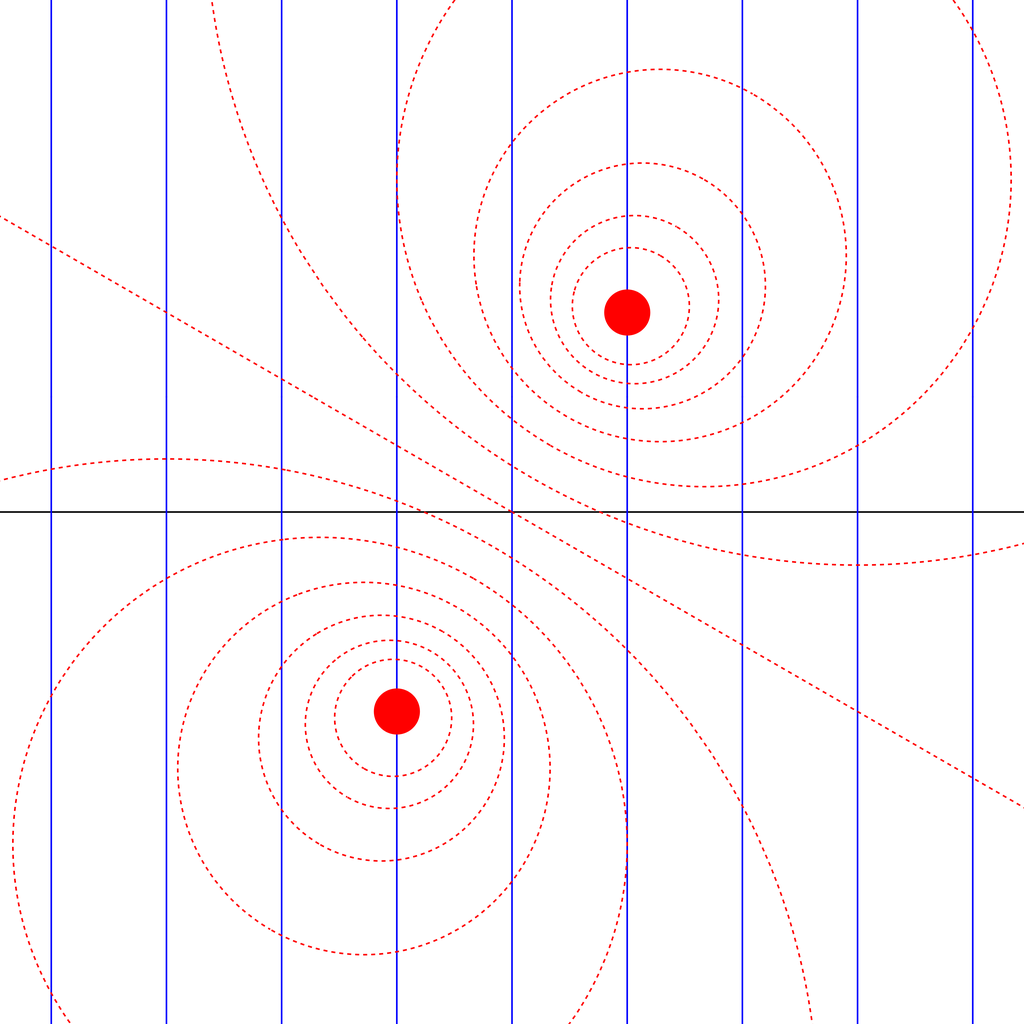}
    \hfill{}
    \includegraphics[width=.35\textwidth]{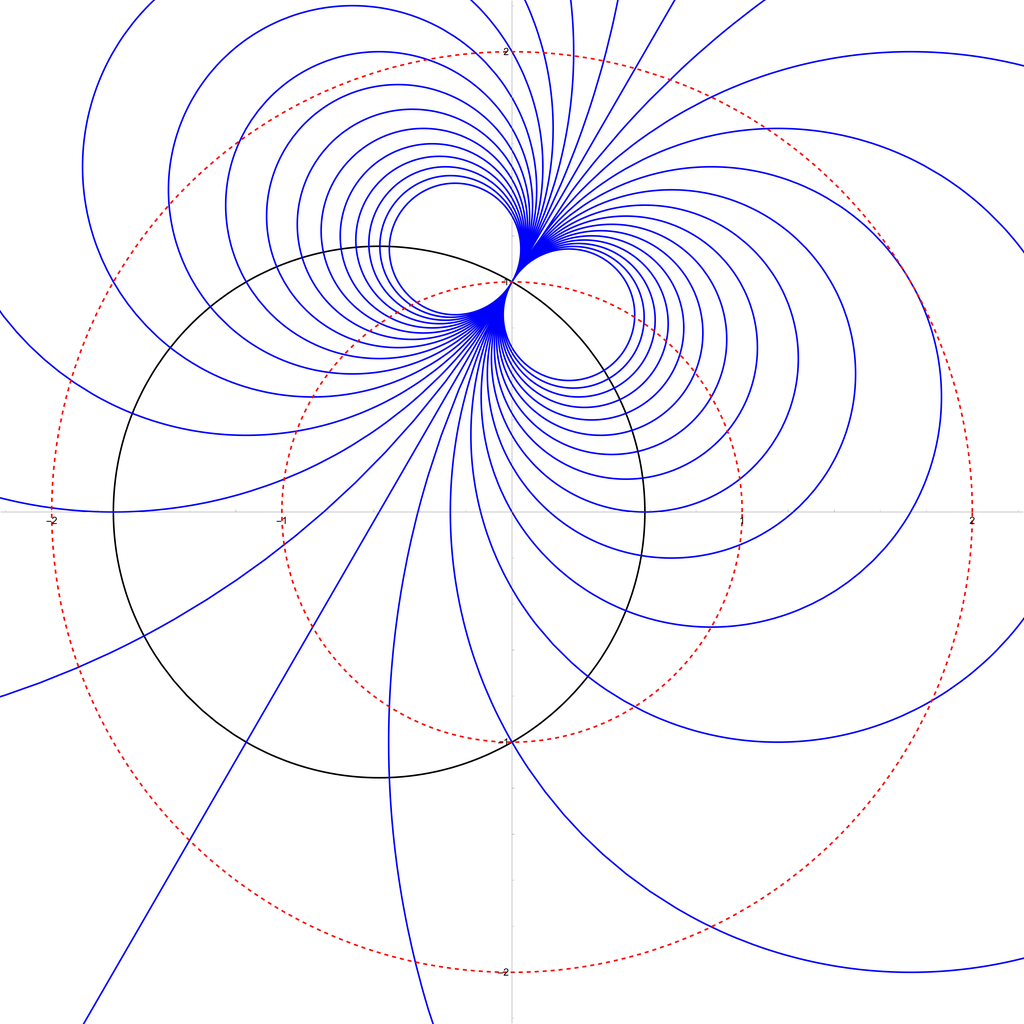}
    \hfill{}
    \caption{Two ways of seeing a M\"obius motion-projection family. In the two pictures, the same projection is shown via its target (black) and fibers (blue). The rotation family is illustrated via its orbits (dashed red). In the picture on the left, the projection is normalized to $(x,y)\mapsto (x,0)$ (at the expense of complicating the motions). In the picture on the right, the rotation family is normalized to $z\mapsto e^{\ii \theta}z$ (at the expense of complicating the projection). The two pictures are related by a M\"obius transformation that sends the red dots on the left to $0$ and $\infty$ on the right and the point $\infty$ to $(0,1)$.}
    \label{fig:twoNormalizations}
\end{figure}

\begin{example}
Lastly, consider a loxodromic motion, shown in Figure \ref{fig:infiniteTangent}. This motion is conjugate, by some $M\in \Mob$, to a mapping of the form $z\mapsto e^{(a+\ii b)t}z$. The orbit of~$\infty$ is non-linear in this case, but transversality nonetheless fails along its tangent line $T$ at~$\infty$. Projection theorems hold in this case, even if the linearization is vertical, since any subset of $T$ moves out of it under the action of $\Gamma$.
\end{example}

\section{Real projective transformations}
\label{sec:season2}

In this section, we repeat the analysis from Section \ref{sec:season1} within the framework of projective geometry. That is, we prove Theorems \ref{thm:RPAllProjections} and \ref{thm:RPSpecialProjections} and then apply Theorem \ref{thm:RPSpecialProjections} to specific projection families. Recall that our basic projection is now the mapping $\pi: (\RP^2\setminus \infty_Y) \rightarrow \RP^1=\R\cup\infty_X$, which restricts to $\R^2$ as the familiar projection $(x,y)\mapsto x$ and also sends the infinite points to $\infty_X$. Writing points of $\RP^2$ and $\RP^1$ in homogeneous coordinates\footnote{A homogeneous coordinate $(x:y:z)\in \RP^2=(\R^3\setminus \{0\})/\R^*$ corresponds either to a finite point $(x/z,y/z)$ if $z\neq 0$, or represents an infinite point if $z=0$. Unlike in the case of the Riemann sphere $\hat \C = \mathbb{CP}^1$, in $\RP^2$ there are infinitely many infinite points. We distinguish two infinite points of interest: $\infty_X=(1:0:0)$ and $\infty_Y=(0:1:0)$. Matrices in $GL(3,\R)$ act on homogeneous coordinates as they would on vectors.}, $\pi$ is given simply by $\pi(x:y:z)=(x:z)$.

Theorem \ref{thm:RPAllProjections} takes an additional consideration not needed for Theorem \ref{thm:MobAllProjections}: namely, the family $O(2)$ is transversal on $\R^2\subset\RP^2$, but not on $\RP^2\setminus \R^2$. Indeed, projection theorems fail for subsets of $\RP^2\setminus \R^2$.

\begin{proof}[Proof of Theorem \ref{thm:RPAllProjections}]
Since the family induced by $O(2)$ is transversal on $\R^2$, we obtain by Proposition \ref{prop:enlargedFamily} that the family induced by $GL(3,\R)$ is transversal on the set $\{(g, p): g^{-1}(p)\in \R^2\}$. 

Let $M=\begin{pmatrix}0&0&-1\\0&1&0\\1&0&0\end{pmatrix}\in GL(3,\R)$, which preserves the point $\infty_Y$ and sends the extended X-axis to itself. In particular, we have that $\pi$ commutes with $M$. Applying Lemma \ref{lemma:transversalitycomposition} to the group $\Gamma=MO(2)M^{-1}$, we have that the family induced by  $\Gamma$ is transversal away from $M(\RP^2\setminus \R^2)$, which is the closure $\hat Y$ of the Y-axis. Combining this with Proposition \ref{prop:enlargedFamily}, we have that the family induced by  $GL(3,\R)$ is transversal on the set $\{(g, p): g^{-1}(p)\in \RP^2 \setminus \hat Y\}$.

Combining the two transversality results, we have that $\Pi:GL(3,\R)\times\RP^2\rightarrow \RP^1$ is transversal away from the set $\{(g,p):g^{-1}p=\infty_Y\}$, as desired.
\end{proof}

The proof of  Theorem \ref{thm:RPSpecialProjections} is analogous to that of Theorem \ref{thm:MobSpecialProjections}, once we prove our next Theorem \ref{thm:RPTransversality}, which parallels Theorem \ref{thm:MobTransversality}. Recall that the Lie algebra $\mathfrak{gl}(3,\R)$  of $GL(3,\R)$ consists of arbitrary 3-by-3 real matrices.

\begin{theorem}
\label{thm:RPTransversality}
Let $\Gamma\subset GL(3,\R)$ be a one-dimensional Lie subgroup parametrized as $\Gamma=\{\gamma_t=\exp(At):t\in\R\}$ for some non-zero $A=(a_{ij})\in \mathfrak{gl}(3,\C)$.  Define:
\begin{itemize}[itemsep=2pt,topsep=1pt]
\item \ $S_0=\{\infty_Y\}\subset  \RP^2$ and $S=\{(\gamma, p): \gamma(p)=\infty_Y\}\subset \Gamma \times \RP^2$.
\item \ $L_0=
\begin{cases}
 \{(x,y): a_{32}x=a_{12}\}\subset \R^2 & \text{ if }a_{32}\neq 0\\
 \RP^2\setminus \R^2 & \text{ if }a_{32}=0 \text{ and }a_{12}\neq 0\\
 \RP^2 & \text{ if }a_{32}=0 \text{ and }a_{12}=0
\end{cases}$
\item \ $L=\{(\gamma, p):\gamma(p)\in L_0\}\subset \Gamma \times  \RP^2$.

\item \ $K_0= L_0\cup S_0\subset  \RP^2$ and $K= L\cup S\subset \Gamma\times \RP^2$.
\end{itemize}
Then the projection family $\Pi:\Gamma\times \RP^2\rightarrow \RP^1$ given by $(\gamma, p)\mapsto \pi(\gamma(p))$ is defined on the domain $(\Gamma\times \RP^2)\setminus S$ and is locally transversal on $(\Gamma\times \RP^2)\setminus K$.
\end{theorem}

\begin{proof}
As in the proof of Theorem \ref{thm:RPAllProjections}, we will first check transversality at finite points following the steps of the proof of Theorem \ref{thm:MobSpecialProjections}. Then we check transversality at the infinite points by linking it to transversality at finite points.

We have $\exp(At)=\id+At+O(t^2)$, where $O(t^2)$ refers to higher-order terms in $t$, so that
\begin{equation}
    \Pi(t,(x,y))=\frac{(1+a_{11}t)x+a_{12}ty+a_{13}t+O(t^2)}{a_{31}tx+a_{32}ty+(1+a_{33})t+O(t^2)}
\end{equation} 
Differentiating at zero, we have
\[\dm_t\vert_{t=0}\Pi(t,(x,y)) = a_{11}x+a_{12}y+a_{13} - x(a_{31}x+a_{32}y+a_{33}).\]
We assume that the left side of the transversality condition hold for a pair of points $(x+\Delta x, y+\Delta y)$ with constant $C>0$. Hence, $\frac{\norm{\Delta x}}{\norm{(\Delta x, \Delta y)}}\leq C$ and $\frac{\norm{\Delta y}}{\norm{(\Delta x, \Delta y)}}$ is large in terms of $C$ (this is analogous to equations \eqref{eq:deltax} and \eqref{eq:deltay}). One can now compute and simplify $ \frac{\norm{\dm_t \vert_{t=0} \Psi(t,v,w)}}{\norm{v-w}}$ analogous to equations \eqref{pretermdots} through \eqref{termdots}. The analog of  the right-hand term here equals
\begin{align}
 (a_{12}-a_{32}x)\tfrac{\norm{\Delta y}}{\norm{(\Delta x, \Delta y)}}
\end{align}
The analog of left-hand term (containing  $ \frac{\norm{\Delta x}}{\norm{(\Delta x, \Delta y)}}$ as a factor) is small compared to $C$ and hence negligible (see equations \eqref{eq:preestimate} and \eqref{eq:estimate}).
Hence, in the finite part of $\RP^2$, we have local transversality at $t=0$ if and only if {$a_{12}-a_{32}x\neq 0$} in $\R^2\subset\RP^2$.

In order to check transversality on the infinite part of $\RP^2$, we conjugate $A$ and $\Gamma$ by the matrix $M$ used in the proof of Theorem \ref{thm:RPSpecialProjections}, giving
\[MAM^{-1}= \begin{pmatrix}{a_{33}}&-a_{32}&-a_{31}\\
-a_{23}&-a_{22}&a_{21}\\
-a_{13}&a_{12}&a_{11}\end{pmatrix}.
\]
The family associated to $M\Gamma M^{-1}$ is then transversal at points $(x',y')$ where $-a_{32}-a_{12} x'\neq 0$. By Lemma \ref{lemma:transversalitycomposition}, local transversality of $\Pi$ at a non-vertical infinite point with homogeneous coordinates $(x:y:0)$ is equivalent to local transversality of the family associated to $M\Gamma M^{-1}$ at $M(x:y:0)=(0:y:-x)=(0:-y/x:1)=(0,-y/x)$. Here, the transversality condition $-a_{32}-a_{12} x'\neq 0$ reduces to $a_{32}\neq 0$, as desired.
\end{proof}

\begin{example}
The classical transversality result for orthogonal linear projections  is encoded by the Lie algebra element $A=\begin{pmatrix}0&-1&0\\1&0&0\\0&0&0\end{pmatrix} = \dm_{t}\vert_{t=0}\begin{pmatrix}\cos t&-\sin t&0 \\\sin{t}&\cos t&0\\0&0&1\end{pmatrix}$. Transversality and projection theorems hold away from the line at infinity, where they both fail.
\end{example}

\begin{example}
Curiously, for the purpose of projection theorems the rotations are equivalent to the group generated by the matrix $A=\begin{pmatrix}0&-1&0\\0&0&0\\0&0&0\end{pmatrix}$, which gives the group of shears $\gamma_t(x,y)=(x+yt, y)$. Again, we have transversality away from the line at infinity.
\end{example}

\begin{example}
Transversality and projection theorems fail globally for all groups that preserve the point $\infty_Y$, since they commute with the projection. These include translations, dilations, and vertical shears. They also include
$Z$-shears of the form $(x,y)\mapsto (\frac{x}{1+xt}, \frac{y}{1+xt})$ and $Z$-rotations of the form $(x,y)\mapsto ( \frac{\cos(t) x-\sin(t)}{\sin(t)x+\cos(t)}, \frac{y}{\sin(t)x+\cos(t)})$.
\end{example}

\begin{example}
\label{ex:point-source}
The family of point-source projections from a finite light source (say, $(0,1)$) can likewise be encoded in our framework by conjugating the standard rotation family $O(2)$ by a mapping that sends the light source to $\infty_Y$, say $N=\begin{pmatrix}1&0&0\\0&0&-1\\0&1&0\end{pmatrix}$. \\ Conjugating the standard rotations by $N$ and differentiating gives the Lie algebra generator $\begin{pmatrix}0&0&-1\\0&0&0\\1&0&0\end{pmatrix}$. One concludes that point-source projections satisfy projection theorems but have artifact non-transversality along the line tangent to the unit circle at rotation angle $\theta=0$, and corresponding points at other values of $\theta$.
\end{example}

\begin{example}
Lastly, we consider the group shown in Figure~\ref{fig:infiniteTangent}, where transversality fails along the linearization of the orbit $\Gamma(\infty_Y)$. The group in this case is conjugate by an element of $GL(3,\R)$ to motions of the form $(x,y)\mapsto (e^{2t}x, e^{3t}y)$, which has non-linear orbits. Conjugating such that one of these orbits passes through $\infty$ provides a desired example.
\end{example}

\section{Spherical and hyperbolic projections}
\label{sec:hypsph}

We now provide transversality results for closest-point projections onto totally-geodesic subspaces in the  hyperbolic space $\Hyp^n$ and sphere $\Sph^n$. 
Previously, transversality for those projections was known for $\Hyp^2$ and $\Sph^2$ \cite{BaloghIseli2016}, and projection theorems \cite{BaloghIseli2018} as well as a weaker version of transversality \cite{AnninaPhD} were known in $\Hyp^n$ for $n\geq 3$ but not for $\Sph^n$.
Our results in $\Sph^n$ are new for $n\geq 2$, and in $\Hyp^n$ we provide a stronger transversality result (higher regularity). \\

We first frame the transversality statements on $\Hyp^n$ and $\Sph^n$ by viewing them as abstract Riemannian manifolds. We will then work with concrete models of $\Hyp^n$ and $\Sph^n$, as is permitted by Lemma \ref{lemma:transversalitycomposition}. Let $X$ be $\Hyp^n$ or $\Sph^n$, and fix a point $p\in X$. The exponential mapping $\exp_p: T_pX\rightarrow X$ maps each $m$-dimensional subspace of $T_pX$ to a totally-geodesic subspace of $X$. Hence, $\exp_p$ induces an identification between the set of $m$-dimensional subspaces of $T_pX$ (that is, the Grassmannian $G(n,m)$) and the set of totally-geodesic subspaces of $X$ containing $p$. By further identifying each element of $G(n,m)$ with $\R^m$ (see \S \ref{sec:prelims}), we may then speak of the family of closest-point projections  $\Pi: G(n,m)\times X \rightarrow \R^m$. Note that closest-point projections are defined globally in $\Hyp^n$. In $\Sph^n$, they are defined away from two points (depending on the element of $G(n,m)$) in the equator $S_0$ dual to $p$.

\begin{theorem}
\label{thm:curvedTransversality}
Let $X$ be hyperbolic space $\Hyp^n$ or the complement $\Sph^n\setminus S_0$ of the equator $S_0$ in the sphere $\Sph^n$, and $m<n$. Then the closest-point projection family $\Pi: G(n,m)\times X \rightarrow \R^m$ is locally transversal and therefore satisfies projection theorems. \end{theorem}

\begin{proof}
In view of Lemma \ref{lemma:transversalitycomposition}, transversality in any model of $X$ passes to all $C^2$-equivalent models, so we may work with any standard model of each space. 

We first consider the sphere $\Sph^n \subset \R^{n+1}$, see Figure \ref{fig:spherical}, with $p$ at the north pole. Recall that totally geodesic subspaces of $\Sph^n$ correspond exactly to intersections of $\Sph^n$ with linear subspaces $V$ of $\R^{n+1}$. Furthermore, given a point $q \in \Sph^n$, one can rotate both $V$ and $q$ into a normalized position to show that the projection of $q$ to $\Sph^n\cap V$ is given by perpendicular projection of $q$ to $V$ followed by rescaling to norm 1.

\begin{figure}
    \centering
    \hfill{}
    \includegraphics[width=.35\textwidth]{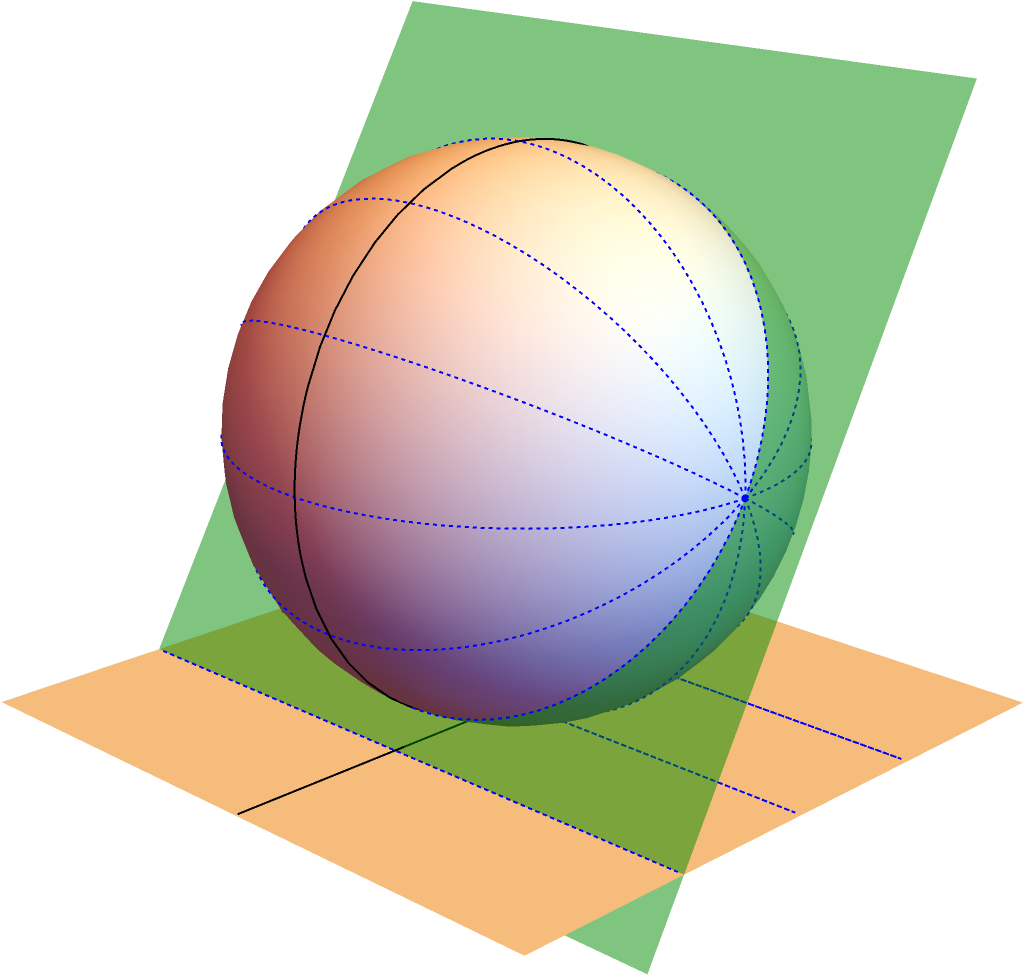}
    \hfill{}
    \includegraphics[width=.35\textwidth]{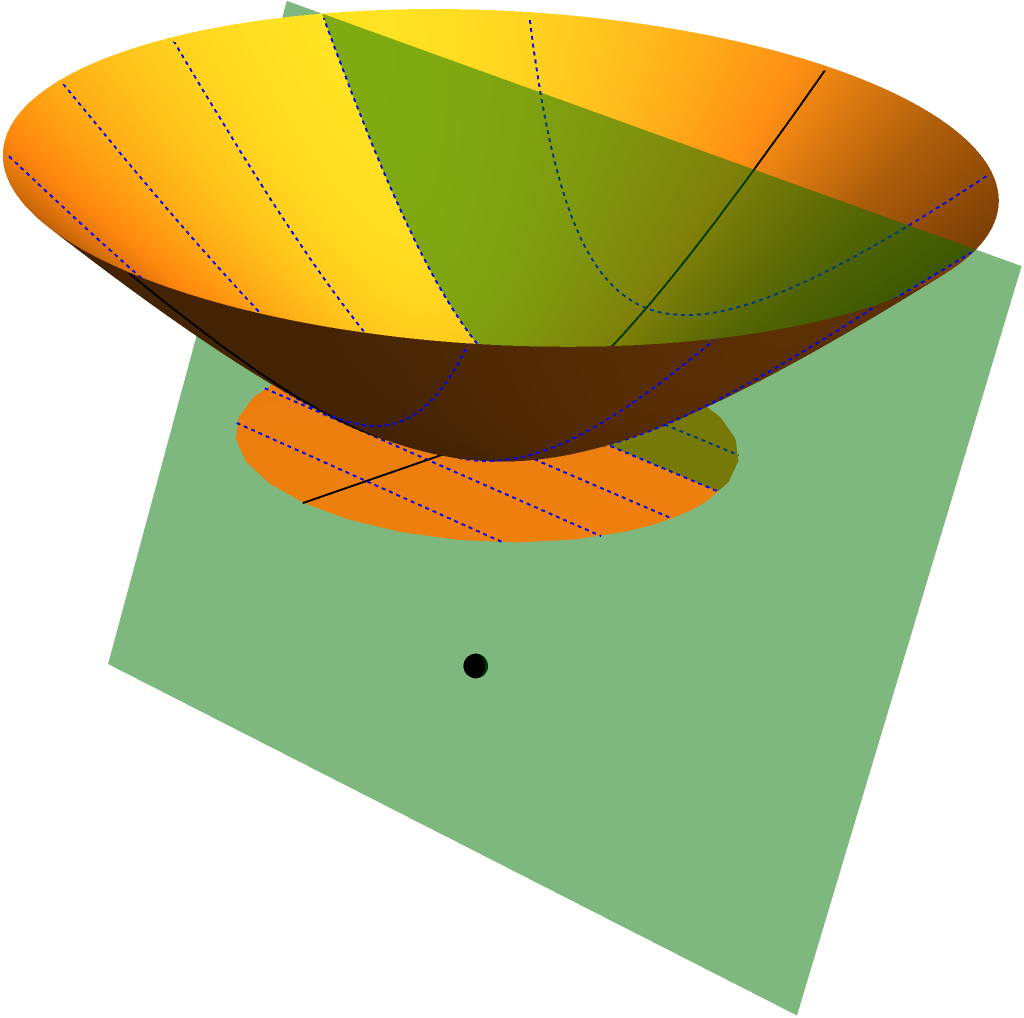}
    \hfill{}
    \caption{Closest-point projections to a totally geodesic subspace (black)  in the sphere and hyperbolic space, shown via the blue dashed fibers, and the corresponding orthogonal projections in $\R^2$. The correspondence is given by radial projection from the origin (center of the sphere on the left, black dot on the right) to the planes $z=-1$ and $z=1$, respectively.
    }
    \label{fig:spherical}
\end{figure}

Let $f(x_1, \ldots, x_{n+1})=(x_1/x_{n+1}, \ldots, x_n/x_{n+1})$ be the projectivization mapping, sending each half-sphere of $\Sph^n\setminus S_0$ diffeomorphically to $\R^n$.
The point $p$ is sent to the origin, and each totally geodesic subspace $V\cap \Sph^n$ of $\Sph^n$ is sent to a linear subspace of $\R^n$, namely to $f(V)=V/\vspan(e_{n+1})$. 
Furthermore, $f$ trivially conjugates the rotation family  $O(n)\times \id$ to the family $O(n)$ in $\R^n$. 

We now show that $f$ conjugates closest-point projections on $\Sph^n$ to orthogonal projections in $\R^n$, even though $f\vert_{\Sph^n}$ is not a conformal mapping. Note first that $f$ is dilation-invariant, so that instead of working with closest-point projections from $\Sph^n$ to $\Sph^n\cap V$, we may work with orthogonal projections from $\R^{n+1}$ to $V$. Next, rotate $V$ using an element of $O(n)\times \id$ so that it is spanned by $e_1, \ldots, e_m$ and $e_{n+1}$, and therefore $f(V)$ is the subspace $\vspan(e_1, \ldots, e_m)$. Perpendicular projection onto $V$ is then given by $(q_1, \ldots, q_{n+1})=(q_1, \ldots, q_m, 0, \ldots, 0, q_{n+1})$.  Conjugating by $f$, we obtain a well-defined mapping
\[(q_1, \ldots, q_n)\mapsto(q_1, \ldots, q_m, 0,\ldots, 0)\]
by choosing any preimage of $q$ under $f$, projecting perpendicularly to $V$, and then applying $f$ again to return to $\R^n$. 
We thus have that, under conjugation by $f$, the family of closest-point projections onto totally geodesic subspaces of dimension $m$ passing through $p$ in $\Sph^n$ is equivalent to the family of orthogonal projections to $m$-planes passing through the origin in $\R^n$. By Lemma \ref{lemma:transversalitycomposition} and transversality of the Euclidean family of projections, we obtain the spherical part of the theorem.

The claim for hyperbolic space could be performed in the same way by starting the hyperboloid model of $\Hyp^n$ and projectivizing to observe the Klein model of $\Hyp^n$.
We instead work directly in the Klein model of $\Hyp^n$ (cf.~the analogous description in \cite{BaloghIseli2018}). That is, we view $\Hyp^n$ as the unit ball in $\R^n$ with a Riemannian metric that is invariant under real linear-fractional transformations preserving the ball. We will be interested in two types of isometries: the rotations $O(n)$ and the hyperbolic transformation given by
\[(x_1,\ldots, x_n)\mapsto \left(\frac{x_1 \cosh(t)+\sinh(t)}{x_1 \sinh(t)+\cosh(t)}, \frac{x_2}{x_1 \sinh(t)+\cosh(t)}, \ldots, \frac{x_n}{x_1 \sinh(t)+\cosh(t)}\right).\]

We start by normalizing the basepoint $p$ to be the origin: starting with any $p$, we can use a rotation in $O(n)$ to normalize $p=(p_1, 0, \ldots, 0)$ and then a hyperbolic transformation to furthermore take $p=0$. It is easy to see that geodesics in $\Hyp^n$ passing through $0$ are straight lines, and therefore each $m$-dimensional totally geodesic subspace in $\Hyp^n$ passing through $0$ is simply a restriction of an $m$-dimensional linear space $W$ passing through $0$. We may therefore identify this set with the Grassmannian $G(n,m)$, with the same action of $O(n)$ in both cases.

It therefore suffices to show that nearest-point projection in $\Hyp^n$ to a subspace $W\cap \Hyp^n$ is given by orthogonal projection. As before, we may normalize $W$ using an element of $O(n)$ so that $W$ is spanned by $e_1, \ldots, e_m$. Given a point $q=(q_1, \ldots, q_n)$ we may rotate along $W$ and perpendicular to it to normalize $q=(q_1, 0, \ldots, 0, q_{m+1}, 0,\ldots, 0)$. Furthermore, we may apply a hyperbolic transformation to normalize $q=(0, \ldots, 0, q_{m+1}, 0, \ldots, 0)$. The line segment joining the origin to $q$ is a geodesic, and so $q$ projects to $0$ under closest-point projection. Note that the last  normalization of $q$ using the hyperbolic transformation does not preserve Euclidean angles, but it does preserve orthogonal projection to $W$, so the closest-point projection to $W\cap \Hyp^n$ is indeed given by orthogonal projection to $W$.

Thus, in the Klein model, we see that the family of closest-point projections to $m$-dimensional totally geodesic subspaces through the origin coincides with the family of Euclidean orthogonal projections onto $m$-dimensional subspaces,  and therefore by Lemma \ref{lemma:transversalitycomposition}  transversality passes over to hyperbolic space, as desired.
\end{proof}

\vspace{6pt}

\subsection*{Acknowledgements} The first author was supported in part by the Swiss National Science Foundation (project no 181898) as well as an AMS Simons Travel Grant.  The second author's travel was supported in part by the U.S.\   NSF grants DMS 1107452, 1107263, 1107367 ``RNMS: Geometric Structures and Representation Varieties'' (the GEAR Network).  Portions of this work were developed during visits by the authors to George Mason University, University of Bern, and University of California, Los Angeles. We thank these institutions for their hospitality. We also thank the referee for carefully reading our article and providing insightful comments.

\bibliographystyle{alpha}
\bibliography{literature_projections}

\newcommand{\etalchar}[1]{$^{#1}$}
\begin{thebibliography}{BDCF{\etalchar{+}}13}

\bibitem[BDCF{\etalchar{+}}13]{BDCFMT2013}
Zolt{\'a}n~M. Balogh, Estibalitz Durand-Cartagena, Katrin F{\"a}ssler, Pertti
  Mattila, and Jeremy~T. Tyson.
\newblock The effect of projections on dimension in the {H}eisenberg group.
\newblock {\em Rev. Mat. Iberoam.}, 29(2):381--432, 2013.

\bibitem[Bes39]{Besic1939}
Abram~S. Besicovitch.
\newblock On the fundamental geometrical properties of linearly measurable
  plane sets of points ({III}).
\newblock {\em Math. Ann.}, 116(1):349--357, 1939.

\bibitem[BFMT12]{BFMT2012}
Zolt{\'a}n~M. Balogh, Katrin F{\"a}ssler, Pertti Mattila, and Jeremy~T. Tyson.
\newblock Projection and slicing theorems in {H}eisenberg groups.
\newblock {\em Adv. Math.}, 231(2):569--604, 2012.

\bibitem[BI16]{BaloghIseli2016}
Zolt{\'a}n~M. Balogh and Annina Iseli.
\newblock Dimensions of projections of sets on {R}iemannian surfaces of
  constant curvature.
\newblock {\em Proc. Amer. Math. Soc.}, 144(7):2939--2951, 2016.

\bibitem[BI19]{BaloghIseli2018}
Zolt\'{a}n~M. Balogh and Annina Iseli.
\newblock Marstrand type projection theorems for normed spaces.
\newblock {\em J. Fractal Geom.}, 6(4):367--392, 2019.

\bibitem[Che18]{Chen2018}
Changhao Chen.
\newblock Restricted families of projections and random subspaces.
\newblock {\em Real Anal. Exchange}, 43(2):347--358, 2018.

\bibitem[Duf18]{dufloux2018linear}
Laurent Dufloux.
\newblock Linear foliations of complex spheres {I}.~{C}hains.
\newblock {\em arXiv preprint 1704.08010}, 2018.

\bibitem[Fal82]{Falconer1982}
Kenneth.~J. Falconer.
\newblock Hausdorff dimension and the exceptional set of projections.
\newblock {\em Mathematika}, 29(1):109--115, 1982.

\bibitem[Fed47]{Federer1947}
Herbert Federer.
\newblock The {$(\varphi,k)$} rectifiable subsets of {$n$}-space.
\newblock {\em Trans. Amer. Soc.}, 62:114--192, 1947.

\bibitem[FO14]{FassOrp2014}
Katrin F\"assler and Tuomas Orponen.
\newblock On restricted families of projections in {$\mathbb R^3$}.
\newblock {\em Proc. Lond. Math. Soc. (3)}, 109(2):353--381, 2014.

\bibitem[Har20]{Harris2020}
Terence L.~J. Harris.
\newblock An a.e. lower bound for {H}ausdorff dimension under vertical
  projections in the {H}eisenberg group.
\newblock {\em Ann. Acad. Sci. Fenn. Math.}, 45(2):723--737, 2020.

\bibitem[Har21]{Harris_Arx2021}
Terrence L.~J. Harris.
\newblock Restricted families of projections onto planes: the general case of
  nonvanishing geodesic curvature.
\newblock {\em arXiv preprint 2107.14701}, 2021.

\bibitem[HJJL12]{HJJL2012}
Risto Hovila, Esa J{\"a}rvenp{\"a}{\"a}, Maarit J{\"a}rvenp{\"a}{\"a}, and
  Fran{\c{c}}ois Ledrappier.
\newblock Besicovitch-{F}ederer projection theorem and geodesic flows on
  {R}iemann surfaces.
\newblock {\em Geom. Dedicata}, 161:51--61, 2012.

\bibitem[Hov14]{Hovila2014}
Risto Hovila.
\newblock Transversality of isotropic projections, unrectifiability, and
  {H}eisenberg groups.
\newblock {\em Rev. Mat. Iberoam.}, 30(2):463--476, 2014.

\bibitem[Ise18]{AnninaPhD}
Annina Iseli.
\newblock Dimension and projections in normed spaces and {R}iemannian
  manifolds.
\newblock {\em PhD thesis, Universit\"at Bern, Switzerland}, 2018.
\newblock Preprint available at:
  \url{http://biblio.unibe.ch/download/eldiss/18iseli_a.pdf}.

\bibitem[JJLL08]{JJLL2008}
Esa J\"arvenp\"a\"a, Maarit J\"arvenp\"a\"a, François Ledrappier, and Mika
  Leikas.
\newblock One-dimensional families of projections.
\newblock {\em Nonlinearity}, 21(3):453--463, 2008.

\bibitem[Kau68]{Kaufman1968}
Robert Kaufman.
\newblock On {H}ausdorff dimension of projections.
\newblock {\em Mathematika}, 15:153--155, 1968.

\bibitem[Kna02]{Knapp2002}
Anthony~W. Knapp.
\newblock {\em Lie groups beyond an introduction}, volume 140 of {\em Progress
  in Mathematics}.
\newblock Birkh\"{a}user Boston, Inc., Boston, MA, second edition, 2002.

\bibitem[Lee13]{Lee2013}
John~M. Lee.
\newblock {\em Introduction to smooth manifolds}, volume 218 of {\em Graduate
  Texts in Mathematics}.
\newblock Springer, New York, second edition, 2013.

\bibitem[Mar54]{Marstrand1954}
John~M. Marstrand.
\newblock Some fundamental geometrical properties of plane sets of fractional
  dimensions.
\newblock {\em Proc. London Math. Soc. (3)}, 4:257--302, 1954.

\bibitem[Mat75]{Mattila1975}
Pertti Mattila.
\newblock Hausdorff dimension, orthogonal projections and intersections with
  planes.
\newblock {\em Ann. Acad. Sci. Fenn. Ser. A I Math.}, 1(2):227--244, 1975.

\bibitem[Mat95]{Mattila1995}
Pertti Mattila.
\newblock {\em Geometry of sets and measures in {E}uclidean spaces}, volume~44
  of {\em Cambridge Studies in Advanced Mathematics}.
\newblock Cambridge University Press, Cambridge, 1995.
\newblock Fractals and rectifiability.

\bibitem[Mat04]{Mattila2004}
Pertti Mattila.
\newblock Hausdorff dimension, projections, and the {F}ourier transform.
\newblock {\em Publ. Mat.}, 48(1):3--48, 2004.

\bibitem[Mat15]{Mattila2015}
Pertti Mattila.
\newblock {\em Fourier analysis and Hausdorff dimension}, volume 150 of {\em
  Cambridge Studies in Advanced Mathematics}.
\newblock Cambridge University Press, Cambridge, 2015.

\bibitem[Mat19]{Mattila2019}
Pertti Mattila.
\newblock Hausdorff dimension, projections, intersections, and besicovitch
  sets.
\newblock In Akram Aldroubi, Carlos Cabrelli, St{\'e}phane Jaffard, and Ursula
  Molter, editors, {\em New Trends in Applied Harmonic Analysis, Volume 2:
  Harmonic Analysis, Geometric Measure Theory, and Applications}, pages
  129--157. Springer International Publishing, Cham, 2019.

\bibitem[OV18]{OrpVen}
Tuomas Orponen and Laura Venieri.
\newblock {Improved Bounds for Restricted Families of Projections to Planes in
  {$\mathbb R^3$}}.
\newblock {\em International Mathematics Research Notices},
  2020(19):5797--5813, 2018.

\bibitem[PS00]{PS2000}
Yuval Peres and Wilhelm Schlag.
\newblock Smoothness of projections, {B}ernoulli convolutions, and the
  dimension of exceptions.
\newblock {\em Duke Math. J.}, 102(2):193--251, 2000.

\bibitem[PSS03]{PeresSimonSolomyak2003}
Yuval Peres, K\'aroly Simon, and Boris Solomyak.
\newblock Fractals with positive length and zero {B}uffon needle probability.
\newblock {\em Amer. Math. Monthly}, 110(4):314--325, 2003.

\end{thebibliography}

\end{document}